\documentclass{article}

\usepackage[centertags]{amsmath}
\usepackage{amssymb,amsfonts}
\usepackage{mathptmx}
\usepackage{mathtools}
\usepackage[colorlinks=true,linkcolor=blue,citecolor=blue]{hyperref}
\usepackage{subfig}
\usepackage{amsthm}
\usepackage{adjustbox}

\usepackage[shortlabels]{enumitem}
\setlist[enumerate]{nosep}

\DeclareMathAlphabet{\mathcal}{OMS}{cmsy}{m}{n}

\usepackage{tikz}
\usetikzlibrary{decorations.pathmorphing}
\usepackage{graphicx, xcolor, soul}

\usepackage[colorinlistoftodos,prependcaption,textsize=tiny]{todonotes}

\usepackage{url}
\usepackage[T1]{fontenc}

\usepackage{tikz}
\usetikzlibrary{calc}

\title{Computable Centering Methods for Spiraling Algorithms and their Duals, with Motivations from the theory of Lyapunov Functions}

\author{Scott B. Lindstrom\\Hong Kong Polytechnic University}

\date{\today}

\def\R{\hbox{$\mathbb R$}}
\def\N{\hbox{$\mathbb N$}}
\def\Newt{\hbox{$\mathcal N$}}

\def\Id{\hbox{\rm Id}}

\newcommand{\argmin}{\ensuremath{\operatorname{argmin}}}
\newcommand{\Fix}{\ensuremath{\operatorname{Fix}}}
\newcommand{\A}{\ensuremath{\mathbb A}}
\newcommand{\B}{\ensuremath{\mathbb B}}
\newcommand{\Cbb}{\ensuremath{\mathbb C}}
\newcommand{\prox}{\operatorname{prox}}

\newcommand{\dom}{\operatorname{dom}}
\newcommand{\aff}{\operatorname{aff}}
\newcommand{\gra}{\operatorname{gra}}

\newtheorem{theorem}{Theorem}
\newtheorem{corollary}{Corollary}
\newtheorem{lemma}{Lemma}
\newtheorem{proposition}{Proposition}
\theoremstyle{definition}
\newtheorem{definition}{Definition}
\newtheorem{example}{\it Example}

\makeatletter
\def\namedlabel#1#2{\begingroup
	#2%
	\def\@currentlabel{#2}%
	\phantomsection\label{#1}\endgroup
}
\makeatother

\begin{document}
	
	\maketitle

\begin{abstract}
	For many problems, some of which are reviewed in the paper, popular algorithms like Douglas--Rachford (DR), ADMM, and FISTA produce approximating sequences that show signs of spiraling toward the solution. We present a meta-algorithm that exploits such dynamics to potentially enhance performance. The strategy of this meta-algorithm is to iteratively build and minimize surrogates for the Lyapunov function that captures those dynamics.
	
	As a first motivating application, we show that for prototypical feasibility problems the circumcentered-reflection method (CRM), subgradient projections, and Newton--Raphson are all describable as gradient-based methods for minimizing Lyapunov functions constructed for DR operators, with the former returning the minimizers of spherical surrogates for the Lyapunov function.
	
	As a second motivating application, we introduce a new method that shares these properties but with the added advantages that it: 1) does not rely on subproblems (e.g. reflections) and so may be applied for any operator whose iterates have the spiraling property; 2) provably has the aforementioned Lyapunov properties with few structural assumptions and so is generically suitable for primal/dual implementation; and 3) maps spaces of reduced dimension into themselves whenever the original operator does. This makes possible the first primal/dual implementation of a method that seeks the center of spiraling iterates. We describe this method, and provide a computed example (basis pursuit).
\end{abstract}

{\small
	
}

\section{Introduction}\label{s:introduction}

Many important optimization problems are computationally tackled by repeated application of an operator $T$, whose fixed points allow one to recover a solution. Under very minimal assumptions, whenever the iterates exhibit local convergence to a fixed point, a particular function $V$ exists. This function $V$, called a Lyapunov function, describes the behaviour of the discrete dynamical system admitted by repeated application of the operator (see \cite[Theorem~2.7]{kellett2005robustness}).

The zeroes that minimize the Lyapunov function are the fixed points for the operator. Thus, if one \textit{did} have some knowledge---or a reasonable guess---about the structure of the Lyapunov function, one could consider the equivalent problem of seeking a minimizer for the Lyapunov function directly. We will introduce a class of methods---based on operators in $\mathbf{MSS}(V)$---that do exactly this, by \textit{minimizing spherical surrogates} for the Lyapunov function.

We suggest implementing such methods when an algorithm's change from iterate to iterate resembles the oscillations in Figures~\ref{fig:ellipse2} and \ref{fig:basis_pursuit}. This phenomenon has been observed for many problems; in addition to the examples in this paper, see also \cite{AB,artacho2019douglas,BCNPW,Benoist,BLSSS,BS,DT,DHL2019,dizon2020centering,giladi2019lyapunov,LLS,liang50improving,LSsurvey,poon2019trajectory}. In such a case, we will say that an algorithm \textit{shows signs of} spiraling. This definition, though informal and subjective, is useful in practice, because this pattern is very easy to check for. For many known examples, this pattern is co-present with a special property (\ref{A1}) that we will \textit{formally} call the \textit{spiraling property}. It is on this latter, formal, difficult-to-check definition that we will build our \textit{theory}, while the former, informal definition is easy to check in \textit{applications}.

Fascinatingly, the \textit{Circumcentered-Reflection Method} (CRM), for certain, highly structured feasibility problems, is an existing example \textit{from} the broader class of algorithms we introduce. However, CRM does \textit{not} succeed for a primal/dual implementation of the basis pursuit problem, for reasons we discover and describe in this paper. Motivated by our theoretical understanding of the new, broader class $\mathbf{MSS}(V)$, we introduce a novel operator (Definition~\ref{def:LT}) that performs very well in our experiments for the primal/dual framework.

\subsection{Background}

Algorithms like the \emph{fast iterative shrinkage-thresholding algorithm} (FISTA)  \cite{poon2019trajectory,liang50improving}, \emph{alternating direction method of multipliers} (ADMM) \cite{boyd2011distributed,eckstein2012augmented,eckstein2015understanding}, and \emph{Douglas--Rachford method} (DR) \cite{LSsurvey,DT,giladi2019lyapunov} seek to solve problems of the form
\begin{equation}\label{objective}
	\underset{x \in E}{\rm minimize}\quad f(x)+g(z) \quad {\rm such\; that }\; Mx=z.
\end{equation}
where $E,Y$ are Hilbert (here Euclidean) spaces, $M:E \rightarrow Y$ is a linear map, and $f,g:E \rightarrow \R \cup \{+\infty \}$ are proper, extended real valued functions. For a common example, when $M=\Id$ is the identity map and $f=\iota_A, g=\iota_B$ with
\begin{equation}\label{def:indicator}
	\iota_S :E \rightarrow \mathbb{R}\cup\{+\infty\} \quad \text{by} \quad \iota_S: x\mapsto \begin{cases}
		0 & \text{if} \; x \in S\\
		+\infty & \text{otherwise}
	\end{cases}
\end{equation}
for closed constraint sets $A,B$ with $A \cap B \neq \emptyset$, \eqref{objective} becomes the \emph{feasibility} problem: 
\begin{equation}\label{feasibility_problem}\tag{FEAS}
	\text{Find} \quad x \in A \cap B,
\end{equation}
For many problems of interest, ADMM, DR, and FISTA exhibit signs of spiraling \cite{LLS,poon2019trajectory,liang50improving}, such as those shown at right in Figure~\ref{fig:ellipse2}.

Borwein and Sims provided the first local convergence result for DR applied to solving the nonconvex feasibility problem \eqref{feasibility_problem} when $A$ is a sphere and $B$ a line \cite{BS}. Arag\'on Artacho and Borwein later provided a conditional global proof \cite{AB} for starting points outside of the axis of symmetry, while behaviour on the axis is described in \cite{BDL18}. Borwein et al. adapted Borwein and Sims' approach to show local convergence for lines and more general plane curves \cite{BLSSS}. In each of these settings, the local convergence pattern resembles the spiral shown at left in Figure~\ref{fig:ellipse2}; Poon and Liang have also documented signs of spiraling in the context of the more general problem \eqref{objective} \cite{poon2019trajectory}. Benoist showed global convergence for DR outside of the axis of symmetry for Borwein and Sims' circle and line problem \cite{Benoist} by constructing the Lyapunov function whose level curves are shown in Figure~\ref{fig:Lyapunovcircle}. Dao and Tam extended Benoist's approach to function graphs more generally \cite{DT}. Most recently, Giladi and R{\"u}ffer broadly used \emph{local} Lyapunov functions to construct a \emph{global} Lyapunov function when one set is a line and the other set is the union of two lines \cite{giladi2019lyapunov}. Altogether, Lyapunov functions have become the definitive approach to proving KL-stability and describing the basins of attraction for DR in the setting of nonconvex feasibility problems.

\begin{figure}
	\begin{center}
		\includegraphics[width=.3\textwidth,angle=90]{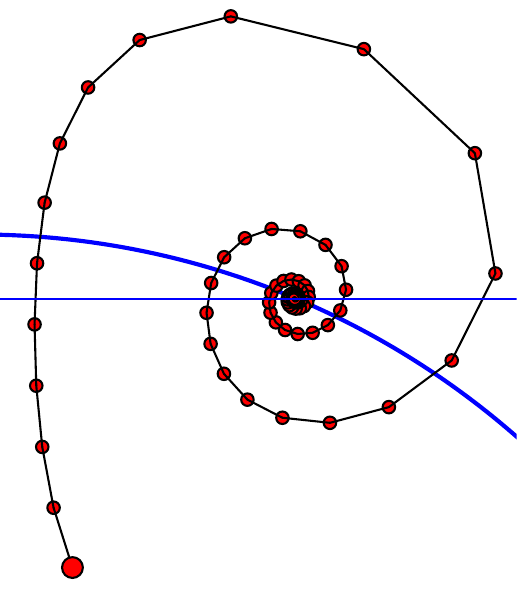}
		\begin{adjustbox}{trim={0.25\width} {0.38\height} {0.25\width} {0.38\height},clip=true}
			\includegraphics[width=1.1\textwidth]{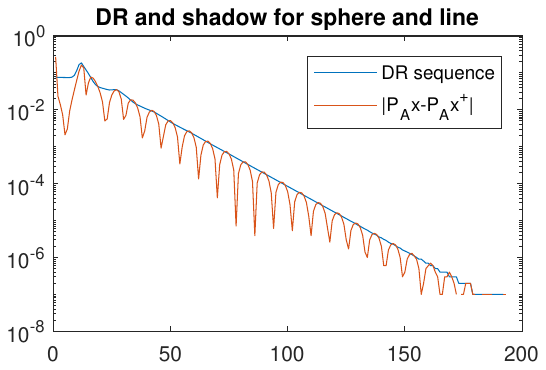}
		\end{adjustbox}
	\end{center}
	\caption{A plot of change from iterate to iterate for shadow sequence for an ellipse and line (right) when governing Douglas--Rachford sequence is spiraling (left).}\label{fig:ellipse2}
\end{figure}

\begin{figure}[t]
	\begin{center}
		\includegraphics[width=.45\textwidth]{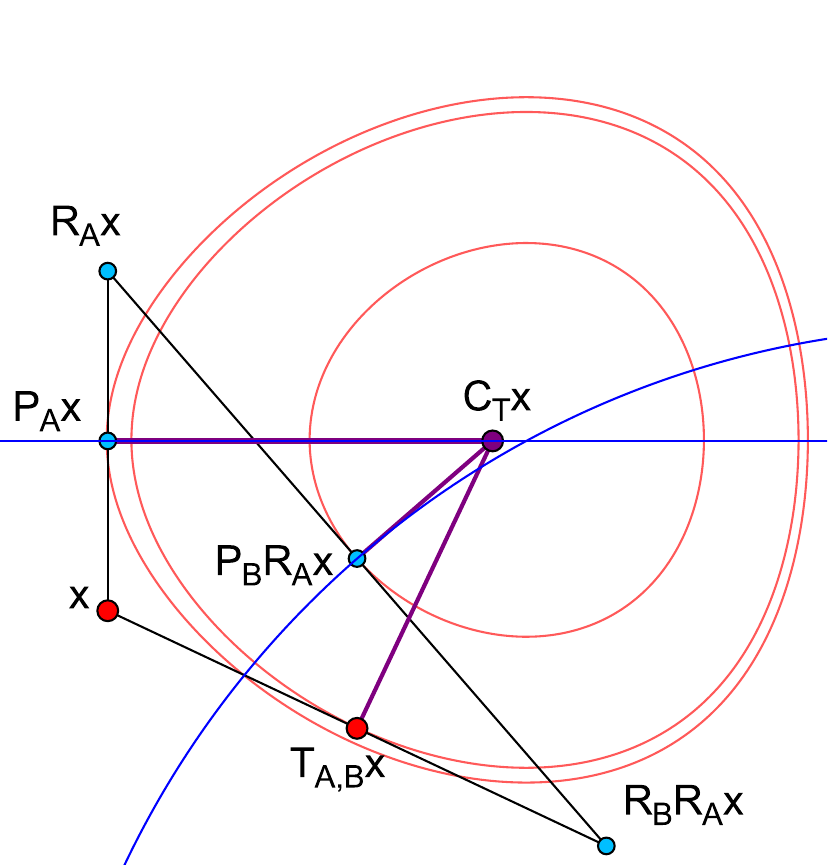}\quad\includegraphics[width=.45\textwidth]{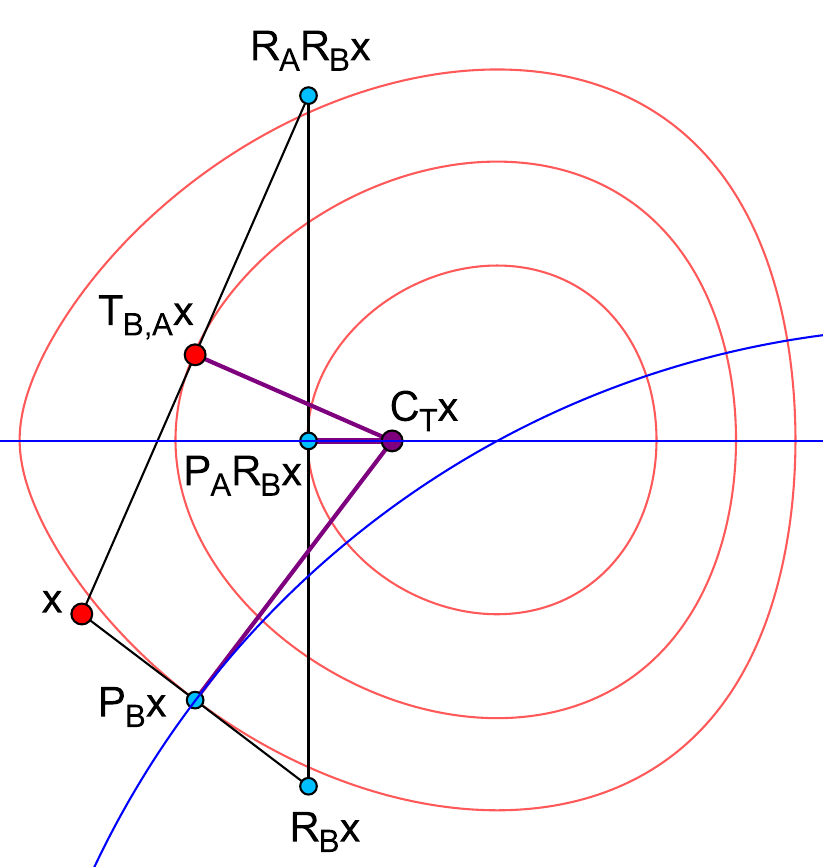}
	\end{center}
	\caption{CRM for $T_{A,B}$ (left) vs CRM for $T_{B,A}$ (right) for circle $B$ and line $A$, together with level curves for Benoist's Lyapunov function.}\label{fig:Lyapunovcircle}
\end{figure}

At the same time, a parallel branch of research has grown from the Douglas--Rachford feasibility problem tree. Behling, Bello Cruz, and Santos introduced what has become known as the circumcentered reflection method (CRM) \cite{bauschke2018circumcentermappings,behling2018linear}. As is locally true of the work of Giladi and R{\"u}ffer, the prototypical setting of Behling, Bello Cruz, and Santos was the convex feasibility problem of finding intersections of affine subspaces. The authors of \cite{DHL2019} illuminated a connection between CRM and Newton--Raphson method, and used it to show super-linear and quadratic local convergence for prototypical problems. Behling, Bello Cruz, and Santos have since used a similar geometric argument to show that the circumcentered reflection method outperforms alternating projections and Douglas--Rachford method for the product space convex feasibility problem \cite{behling2019convex}. See also \cite{arefidamghani2021circumcentered}.

\subsection{Contributions and outline}\label{sec:contributions}

In Sections~\ref{s:introduction} and \ref{s:preliminaries}, we introduce preliminaries.

In Section~\ref{s:CRM}, we show that for many prototypical feasibility problems for which the Lyapunov functions for the Douglas--Rachford dynamical system are known, CRM may be characterized as a \emph{gradient descent} method applied to the Lyapunov function with a special step size (Theorem~\ref{thm:gradientdescent}, Corollary~\ref{cor:gradient}). We also uncover Lyapunov function gradient relationships for subgradient descent methods and Newton--Raphson method (Proposition~\ref{prop:NewtonRaphson}). 

In Section~\ref{s:LT}, we introduce the class of operators $\mathbf{MSS}(V)$ that minimize spherical surrogates for Lyapunov functions. We then show that for the class of problems considered in Section~\ref{s:CRM}, CRM actually returns the minimizer of a \textit{spherical} \textit{surrogate} for the Lyapunov function (Theorem~\ref{thm:quadratic}). We then introduce a new operator (Definition~\ref{def:LT}) that belongs to $\mathbf{MSS}(V)$ with very few assumptions (Theorem~\ref{thm:LT}). We design this new operator to depend only on the governing sequence. This means it can be implemented in a much more general setting than feasibility problems with reflection substeps. For example, one could apply it with black box applications when only the governing sequence is known and subproblem solutions (e.g. proximal points, reflections) are not available. Naturally, this also means that the new method can be used in circumstances when algorithm parameters are fixed (e.g. black box scenarios) or a theoretically optimal choice for them is unknown (such as \cite{dizon2020centering,LLS}). 

In Section~\ref{s:duality}, we apply this new operator to the dual of ADMM for the basis pursuit problem (Section~\ref{s:duality}). In this setting, its lack of dependence on substeps allows it to succeed where a generalization of CRM fails (for reasons we describe). This is also the \emph{first} primal/dual framework and implementation for a method that uses the circumcenter. For warm-started applications (e.g. optimal power flow), one might expect the starting point of iteration to lie near or within the local basin of attraction to a fixed point, so that steps obtained by our approach---minimizing a surrogate for the Lyapunov function for the \textit{dual} iterates---may be the preferred updates from early on in the computation. For continuous optimization problems, one can use objective function checks to choose between Lyapunov surrogate updates and regular updates, as we do for our basis pursuit example. One can compute both such updates in parallel. 

In Section~\ref{s:conclusion}, we conclude with suggested further research.

\section{Preliminaries}\label{s:preliminaries}

The following introduction to the Douglas--Rachford method is quite standard. We closely follow \cite{DHL2019}, which is an abbreviated version of that found in the survey of Lindstrom and Sims \cite{LSsurvey}. 

\subsection{The Douglas--Rachford Operator and Method}\label{ss:DR}

The problem \eqref{objective} is frequently presented in the slightly different form
\begin{equation}\label{eqn:operator_sum_problem}\tag{P}
	\text{Find}\quad x \quad \text{such that}\quad 0 \in (\A+\B)x,
\end{equation}
where $\A$ and $\B$ are often (not always) maximally monotone operators. When the operators are subdifferential operators $\partial f$ and $\partial (g\circ M)$ for the convex functions $f$ and $g\circ M$, one recovers \eqref{objective}. Whenever a set $S$ is closed and convex, its indicator function $\iota_S$, defined in  \eqref{def:indicator}, is lower semicontinuous and convex, while its subdifferential operator %
is the normal cone operator of set $S$. The resolvent $J_{\gamma F} := (\Id+\gamma F)^{-1}$ for a set-valued mapping $F$
generalizes the proximity operator
\begin{equation}\label{def:prox}
	\prox_{\gamma f}(x) = \underset{z \in \R^n}{\argmin}\left( f(z)+\frac{1}{2\gamma} \|x-z\|^2\right) = J_{\gamma \partial f}(x)
\end{equation}
In particular, the resolvent $J_{N_S}$ of the normal cone operator $N_S$ for a closed set $S$ is simply the projection operator given by
$$\mathbb{P}_S(x) := \left \{ z \in S : \|x - z\| = \inf_{z' \in S}\|x - z'\|\right \}=\prox_{\iota_S}(x).$$
When $S$ is nonconvex, $\mathbb{P}_S$ is generically a set-valued map whose images may contain more than one point. For the prototypical problems we discuss, $\mathbb{P}_S$ is always nonempty, and we work with a selector 
$P_S:E \rightarrow S: x \mapsto P_S(x) \in \mathbb{P}_S(x)$.
The Douglas--Rachford method is defined by letting $\lambda>0, x_0\in X$, and setting
\begin{align}
	x_{n+1}\in Tx_n\;\; \text{where}\;\; T_{\A,\B}:X\rightarrow X :x&\mapsto \frac{1}{2}R_{\B}^\lambda R_{\A}^\lambda x + \frac12 x,\;\; \text{and}\;\; R_{\Cbb}^\lambda := 2J_{\Cbb}^\lambda-\Id \label{LM}
\end{align}
is the reflected resolvent operator. Though many newer results exist, the classical convergence result for the Douglas--Rachford method was given by Lions and Mercier \cite{LM} and relies on \textit{maximal monotonicity} \cite[20.2]{BC} of $\A,\B,\A+\B$. 
When the operators $\A$ and $\B$ are the normal cone operators $N_A=\partial \iota_A$ and $N_B=\partial \iota_A$, the associated resolvent operators $J_{N_A}$ and $J_{N_B}$ are the proximity operators $\prox_{\iota_A}$ and $\prox_{\iota_B}$, which may be seen from \eqref{def:prox} to just be the projection operators $P_A$ and $P_B$ respectively. In this case, \eqref{LM} becomes
\begin{align}\label{def:RRA}
	T_{A,B}:x \rightarrow \frac{1}{2}R_BR_Ax + \frac12 \Id x,\quad \text{where}\quad R_S := 2P_S - \Id,
\end{align}
is the \emph{reflection} operator for the set $S$. The operator described in \eqref{def:RRA} for solving \eqref{feasibility_problem} is a special case of the Douglas--Rachford operator described in \eqref{LM} for solving \eqref{objective}. One application of the operators $T_{A,B}$ and $T_{B,A}$ is shown in Figure~\ref{fig:Lyapunovcircle}, where $B$ is a circle and $A$ is a line. From this picture and from \eqref{def:RRA}, one may understand why DR is also known as \emph{reflect-reflect-average} (see \cite{LSsurvey} for other names).


The final averaging step apparent in the form of \eqref{LM} serves to make the operator $T$ firmly nonexpansive in the convex setting \cite{BC}, but Eckstein and Yao actually note that this final averaging step of blending the identity with the nonexpansive operator $R_BR_A$ serves the important geometric task of ensuring that the dynamical system admitted by repeated application of $T_{A,B}$ does not \emph{merely} orbit the fixed point at a constant distance without approaching it \cite{eckstein2012augmented}. In this sense, the tendency of splitting methods to spiral actually motivates the final step of construction for DR: averaging may be viewed as a centering method. It is a safe centering method in the variational sense that it adds theoretically advantageous nonexpansivity properties to the operator $T$ rather than risking those properties in the way that a more bold centering step, like circumcentering, does. 

While fixed points may not be feasible, they allow for quick recovery of feasible points, since
$(x \in \Fix T_{A,B}) \implies P_A x \in A \cap B$ (see, for example, \cite{LSsurvey}). Because of this, the sequence $P_A x_n$ is sometimes referred to as the shadow sequence of $x_n$ (on $A$). While fixed points may not be feasible, they allow for quick recovery of feasible points, since $(x \in \Fix T_{A,B}) \implies P_A x \in A \cap B$ (see, for example, \cite{LSsurvey}). Similarly, for the more general operator $T_{\A,\B}$ in \eqref{LM}, maximal monotonicity of $\A$ and $\B$ is sufficient to guarantee that the points in $J_{\A}(\Fix T_{\A,\B})$ solve \eqref{eqn:operator_sum_problem}; see \cite[Proposition~26.1]{BC}.

While the convergence of DR for convex problems is well-known, the method also solves many nonconvex problems. In addition to \cite{LSsurvey}, we refer the interested reader to the excellent survey of Arag{\'o}n Artacho, Campoy, and Tam \cite{artacho2019douglas}. Li and Pong have also provided some local convergence guarantees for the more general optimization problem \eqref{objective} in \cite{LP}. Critically for us, DR also has a dual relationship with ADMM, which we will describe and exploit in Section~\ref{s:duality}.

\subsection{Lyapunov functions and stability}\label{s:stability}

This abbreviated introduction to Lyapunov functions follows those of greater detail in the works of Giladi and R{\"u}ffer \cite{giladi2019lyapunov} and of Kellett and Teel \cite{kellett2003advances,kellett2005robustness}. 

Let $U \subset E$ where $E$ is a Euclidean space, $T:U \rightrightarrows U$ be a set-valued operator, and define the difference inclusion:
$x_{n+1} \in Tx_n, \quad n \in \N$.
We say that $\varphi(x_0,\cdot):\N \rightarrow E$ is a \emph{solution} for the difference inclusion, with initial condition $x_0 \in U$, if it satisfies
\begin{align*} \varphi(x_0,0)=x_0\quad\text{and}\quad(\forall n \in \N)\quad \varphi(x_0,n+1) \in T(\varphi(x_0,n)).
\end{align*}

\begin{definition}[Lyapunov function {\cite[Definition~2.4]{kellett2005robustness}}]\label{def:Lyapunov}
	Let $\omega_1,\omega_2:E \rightarrow \R_+$ be continuous functions and $T:U \rightrightarrows U$. A function $V:U \rightarrow \R_{+}$ is said to be a Lyapunov function with respect to $(\omega_1,\omega_2)$ on $U$ for the difference inclusion $x^+ \in Tx$, if there exist continuous, zero-at-zero, monotone increasing, unbounded $\alpha_1,\alpha_2:\R_{+}\rightarrow \R_{+}$ such that for all $x \in U$,
	\begin{alignat*}{1}
		&\alpha_1(\omega_1(x)) \leq V(x) \leq \alpha_2(\omega_2(x)),\quad \underset{x^+ \in Tx}{\sup} \; V(x^+) \leq V(x)e^{-1},\quad \text{and}\\
		&V(x)=0\;\; \iff \;\; x \in \mathcal{A}\quad \text{where}\\
		&\mathcal{A}:=\bigg \{\xi \in U: \underset{k \in \N,\; (\forall j\in \{0,\dots,k-1\})\; \varphi(\xi,j+1) \in T(\varphi(\xi,j))}{\sup} \omega_1(\varphi(\xi,k))=0 \bigg \}.
	\end{alignat*}
	When $\omega_1=\omega_2=\omega$, $\mathcal{A}$ is closed and 
	$\mathcal{A}:=\{x \in U\;:\; \omega(x)=0  \}$.
\end{definition}
We forego the usual definition of \emph{robust} $\mathcal{KL}$-stability (see, for example, \cite[Definition~2.3]{kellett2005robustness} or \cite[Definition~3.2]{giladi2019lyapunov}) in favor of remembering the contribution of Kellett and Teel \cite[Theorem~2.10]{kellett2005robustness} that robust stability is equivalent to $\mathcal{KL}$-stability when $T$ satisfies certain conditions on $U$.

\begin{theorem}[Existence of a Lyapunov function {\cite[Theorem~2.7]{kellett2005robustness}}]\label{thm:existence}
	Let $T:U \rightrightarrows U$ be upper semicontinuous on $U$ (in the sense of \cite[Definition~2.5]{kellett2005robustness}), and $T(x)$ be nonempty and compact for each $x \in U$. Then, for the difference inclusion $x^+ \in Tx$, there exists a smooth Lyapunov function with respect to $(\omega_1,\omega_2)$ on $U$ if and only if the inclusion is robustly $\mathcal{KL}$-stable with respect to $(\omega_1,\omega_2)$ on $U$.
\end{theorem}
For us, the principal importance of this result is captured by Giladi and R{\"u}ffer's summary that: ``in essence, asymptotic stability implies the existence of a Lyapunov function.'' They used a related result to guarantee robust $\mathcal{KL}$-stability for their setting, by means of constructing the prerequisite Lyapunov function \cite{giladi2019lyapunov}. Similarly, Benoist and Dao and Tam constructed Lyapunov functions to show KL-stability in their respective settings \cite{Benoist,DT}.



\subsection{The circumcentering operator and algorithms}\label{ss:circumcentering_introduction}

Behling, Bello Cruz, and Santos introduced the circumcentered-reflections method (CRM) for feasibility problems involving affine sets \cite{circumcentering}. The idea is to update by
\begin{equation*}
	C_T(x):=C(R_B R_A(x),R_A(x),x)
\end{equation*}
where $C(x,y,z)$ denotes the point equidistant to $x,y,z$ and lying on the affine subspace defined by them: $\aff\{x,y,z\}$. If $x,y,z$ are not colinear (so their affine hull $\aff\{x,y,z\}$ has dimension $2$) then $C(x,y,z)$ is the center of the circle containing all three points. If $\{x,y,z\}$ has cardinality $2$, $C(x,y,z)$ is the average (the midpoint) of the two distinct points. If $\{x,y,z\}$ has cardinality $1$, $C(x,y,z)=x=y=z$. The circumcenter of three points is easy to compute; for a simple expression, see \cite[Theorem~8.4]{bauschke2018circumcenters}.

When $R_B R_A(x),R_A(x),x$ are distinct \emph{and} colinear, then $(R_B R_A(x),R_A(x),x) \notin \dom C$ and so $x \notin \dom C_T$. When $\dom C_T=E$, we say $C_T$ is \emph{proper}; this is the case in Behling, Bello Cruz, and Santos' prototypical setting of affine sets  \cite{circumcentering,behling2018linear}. Bauschke, Ouyang, and Wang have provided sufficient conditions to ensure that $C_T$ is proper \cite{bauschke2018circumcentermappings,bauschke2018circumcenters}.

Because $C_T$ is not generically proper, the authors of \cite{DHL2019} suggested the modest piecewise remedy of iterating
\begin{equation}\label{def:circumcenteredDR}
	\mathcal{C}_{T_{A,B}}:E \rightarrow E: \quad x\mapsto \begin{cases}
		T_{A,B}x & \text{if}\; x,\;R_Ax,\;\text{and}\;R_BR_Ax\;\;\text{are colinear}\\
		C_Tx & \text{otherwise}
	\end{cases}.
\end{equation}
$\mathcal{C}_{T_{A,B}}$ specifies to $C_T$ when $C_T$ is proper---except in the case $R_BR_Ax = x \neq R_Ax$, in which case $P_Ax \in \Fix T_{A,B}$ solves \eqref{feasibility_problem}.


\section{Spiraling, and subdifferentials of Lyapunov functions}\label{s:CRM}

We next show how CRM is related to known Lyapunov functions for nonconvex Douglas--Rachford iteration.

\subsection{Known Lyapunov constructions}

For the sphere and line, Benoist remarked that the spiraling apparent at left in Figure~\ref{fig:ellipse2} suggests the possibility of finding a suitable Lyapunov function that satisfies
\begin{itemize}
	\item[\namedlabel{A1}{\textbf{A1}}] $(\forall x \in U) \quad \langle \nabla V(x^+), x-x^+ \rangle = 0$.
\end{itemize}
The condition \ref{A1} is visible in Figure~\ref{fig:Lyapunovcircle}, where the level curves of the Benoist's Lyapunov function are tangent to the line segments connecting $x$ and $x^+$. Whenever \ref{A1} holds, we will say that the algorithm $x^+ \in Tx$ is \textit{spiraling} on $U$. In order to reformulate \ref{A1} as an explicit differential equation that may be solved for $V$, one must invert the operator $T$ locally. For this reason, it is often difficult to show \ref{A1} explicitly. Notwithstanding, this \textit{spiraling} property has been present for both the Lyapunov constructions of Dao and Tam \cite{DT} and of Giladi and R{\"u}ffer \cite{giladi2019lyapunov}.

Dao and Tam's results are broad, but we are most interested in a lemma they provided while outlining a generic process for building a Lyapunov function $V$ that satisfies \ref{A1} for $x^+ \in T_{A,B}x$. Specifically, we consider the case when $f:X \rightarrow \left ] -\infty,+\infty \right ]$ is a proper function on a Euclidean space $X$ with closed graph 
\begin{equation}\label{eqn:AB}
	B=\gra f= \{(y,f(y))| y \in \dom f\subset X \}\quad \text{and}\quad A=X\times \{0\}.
\end{equation}
In what follows, $E=X \times \R$, while $\partial f$ and $\partial^0f$ respectively denote the \emph{limiting subdifferential} and the \emph{symmetric subdifferential} of $f$; see, for example, \cite[2.3]{DT}.
\begin{lemma}[{\cite[Lemma~5.2]{DT}}]\label{lem:tangency}
	Suppose that $D\subset \dom f$ is convex and nonempty, and that $F:D \rightarrow \left]-\infty,+\infty\right]$ is convex and satisfies
	\begin{align}
		&(\forall y \in D)\;\; \partial F(y) \supset \begin{cases} 
			\left\{\frac{f(y)}{\|y^*\|^2}y^*\;|\;y^*\in \partial^0f(y) \right\} & \text{if}\;0 \notin \partial^0 f(y)\label{lem:F} \\
			\{0\} &\text{if}\;f(y)=0
		\end{cases}\\
		\text{and}\quad&V:D\times \R \rightarrow \left[-\infty,+\infty \right]:(y,\rho) \mapsto F(y)+\frac{1}{2}\rho^2.\label{lem:V}
	\end{align}
	Let $x=(y,\rho) \in \dom f \times \R$ and $x^+=(y^+,\rho^+) \in T_{A,B}(y,\rho)$. Then the following hold.
	\begin{enumerate}[label=(\alph*)]
		\item\label{lem:tangent_curve} $V$ is a proper convex function on $D \times \R$ whose subdifferential is given by
		\begin{equation}\label{V:subdiff}
			(\forall y \in \dom \partial F)\quad \partial V(y,\rho)=\partial F(y)\times \{\rho \}.
		\end{equation}
		\item\label{lem:tangent_x+} Suppose either $y^+ \in D \setminus (\partial^0f)^{-1}(0)$ and $f$ is Lipschitz continuous around $y^+$, or that $y^+ \in D \cap f^{-1}(0)$. Then there exists $\bar{x}^+ = (\bar{y}^+,\rho^+) \in \partial V(x^+)$ with $\bar{y}^+ \in \partial F(y^+)$ such that
		\begin{equation}\label{V:subdiff2}
			\langle \bar{x}^+,x-x^+ \rangle = \langle (\bar{y}^+,\rho^+),(y,\rho)-(y^+,\rho^+) \rangle = 0.
		\end{equation}
	\end{enumerate}
\end{lemma}
The next example is a special case of \cite[Example~6.1]{DT} and illustrates Lemma~\ref{lem:tangency}.
\begin{example}[Constructing $V$ for $T_{A,B}$ when $A,B$ are lines {\cite[Example~6.1]{DT}}]\label{ex:2lines}
	Let $f:\R \rightarrow \R$ be the linear function $f:y\rightarrow \arctan(\theta) y$ for some $\theta \in \left]-\pi/2,\pi/2 \right[$. Then \eqref{lem:F} amounts to $\nabla F(y) = f(y)/\nabla f(y)=y$, and so $F:y \rightarrow y^2/2+c$ and $V(y,\rho) =(y^2+\rho^2)/2+c/2$. 
\end{example}

\begin{figure}[t]
	\begin{center}
		\includegraphics[width=.45\textwidth]{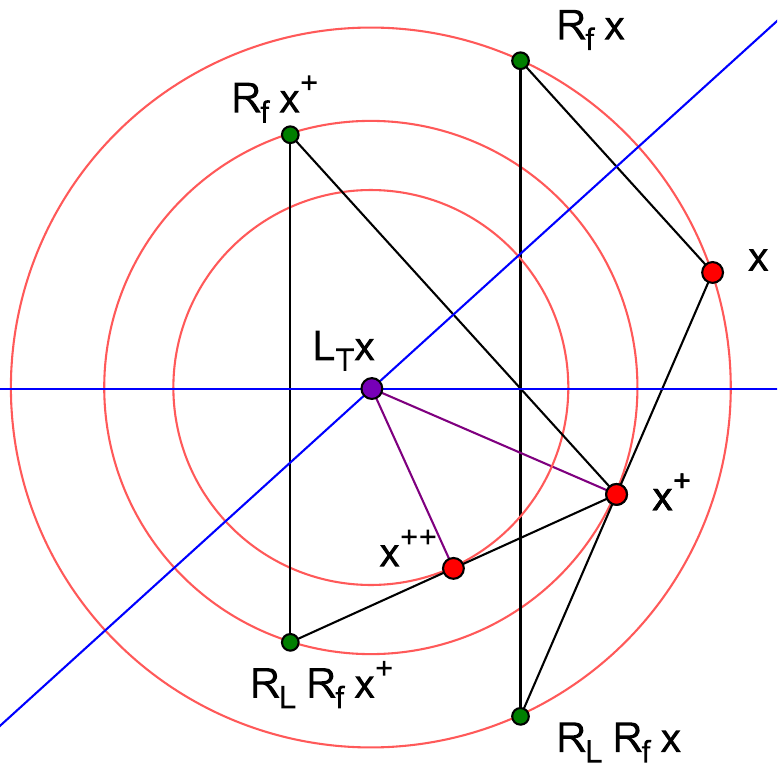}\quad \includegraphics[width=.50\textwidth]{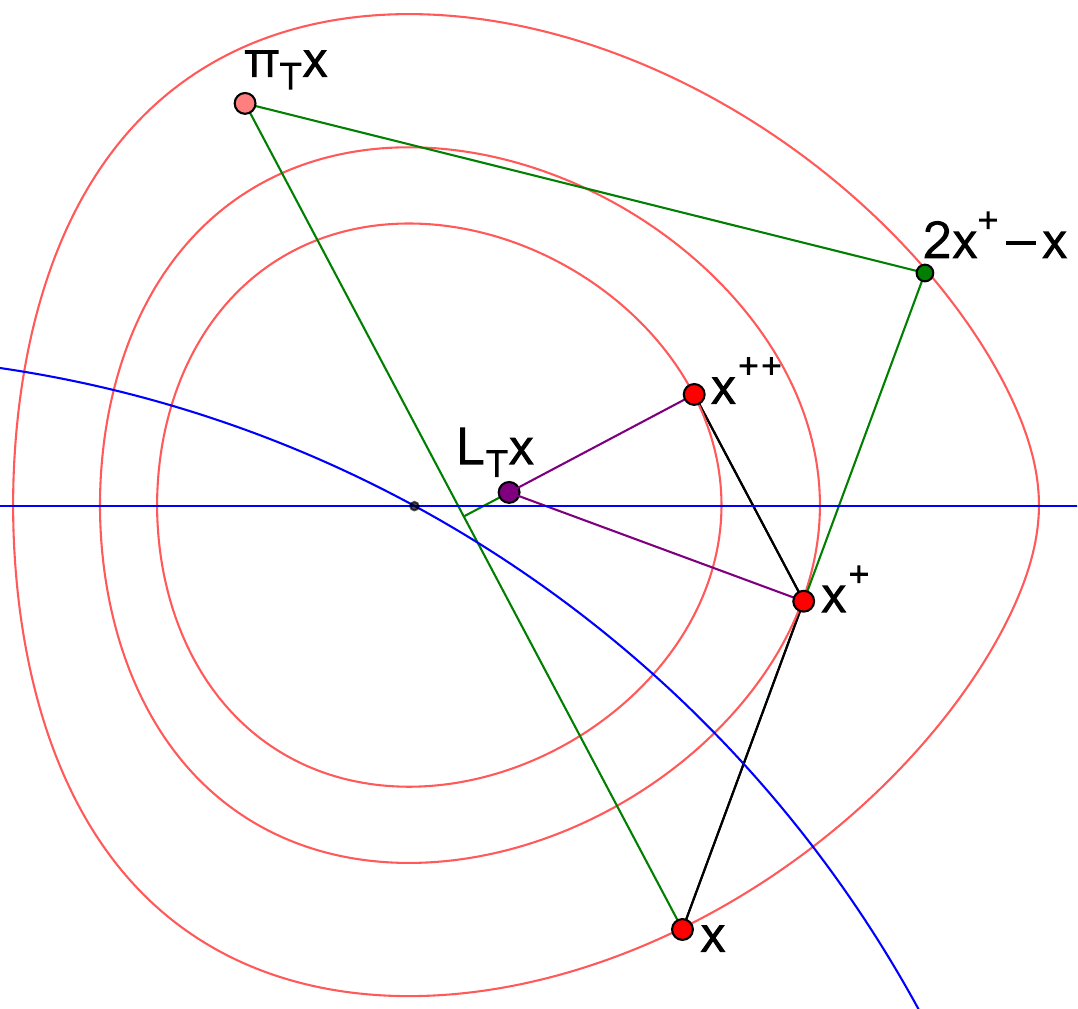}
	\end{center}
	\caption{Construction of $L_{T_{B,A}}$ for Example~\ref{ex:2lines} (left), and for the circle $B$ and line $A$  (right), together with level curves for Lyapunov functions.}\label{fig:lines_and_quadratic}
\end{figure}

Example~\ref{ex:2lines}, for which the modulus of linear convergence of $T_{A,B}$ is known to be $\cos(\theta)$ \cite{BCNPW} (see also \cite[Figure~2]{giladi2019lyapunov}), is shown in Figure~\ref{fig:lines_and_quadratic} (left), where we abbreviate $\gra f$ by $f$. In the case when $A$ is the union of two lines and $B$ is a third line, Giladi and R{\"u}ffer used two local quadratic Lyapunov functions of this type $V_i(x)=\|x-p_i\|^2, i=1,2,$ \cite[Proposition~5.1]{giladi2019lyapunov} for the two respective feasible points $p_1,p_2 \in A \cap B$ in their construction of a global Lyapunov function \cite[Theorem~5.3]{giladi2019lyapunov}.

\subsection{Circumcentered Reflections Method and Gradients of $V$}

Now we will explore the relationship between circumcentering methods and the underlying Lyapunov functions that describe local convergence of the operator $T$. The following definition will simplify the exposition.

\begin{definition}
	For $y,z \in E$ we define
	\begin{equation*}
		H(y,z) := \frac{y+z}{2} + \{y-z\}^\perp
	\end{equation*}
	to be the perpendicular bisector of the line segment adjoining $y$ and $z$. 
\end{definition}	
We will also make use of the following properties (which have been used in various places; for example, \cite{AB,DT}).
\begin{proposition}\label{prop:handy}
	Let  $A,B$ be as in \eqref{eqn:AB}, and let $D,V,F$ be as in Lemma~\ref{lem:tangency}. Then the following hold.
	\begin{enumerate}[label=(\roman*)]
		\item\label{p:indirectmotion} $R_A$ is an indirect motion in the sense that it preserves pairwise distances\footnote{To see why distances are preserved, set $x_1=x_2=x_3-x_4$ and notice that $\langle x_1,x_2 \rangle = \langle R_A x_1,R_Ax_2 \rangle$ becomes $\|x_3-x_4\|=\|R_A(x_3-x_4)\|=\|R_Ax_3-R_Ax_4\|$ where the final equality simply uses the linearity of $R_A$} and angles. In other words:
		\begin{equation*}
			(\forall x_1,x_2 \in X \times \R)\quad \langle x_1,x_2 \rangle = \langle R_Ax_1,R_Ax_2 \rangle.
		\end{equation*}
		\item\label{p:Vsymmetry} $V$ is symmetric about $A$ on $D\times \R$ in the sense that
		\begin{subequations}
			\begin{align}
				&(\forall (y,\rho) \in D \times \R) \quad V(y,\rho)=V(y,-\rho),\label{p:Vsymmetry1}\quad \text{and\;so}\quad V=V\circ R_A\\
				&\text{and} \quad (y^*,\rho^*) \in \partial V(y,\rho) \;\; \iff \;\; (y^*,-\rho^*) \in \partial V(y,-\rho),\label{p:partialVsymmetry}\\
				&\text{which\;allows\;us\;to\;write}\quad R_A \partial V = \partial V \circ R_A.\label{p:partialVsymmetry2}
			\end{align}
		\end{subequations}
		\item\label{p:Tequivalence} $T_{B,A}$ and $T_{A,B}$ are equivalent across the axis of symmetry $A$ in the sense that
		\begin{equation*}
			T_{B,A} = R_A T_{A,B} R_A\quad \text{and}\quad T_{A,B} = R_AT_{B,A}R_A
		\end{equation*}
	\end{enumerate}
\end{proposition}
\begin{proof}
	\ref{p:indirectmotion}: Let $x_1=(y_1,\rho_1),x_2=(y_2,\rho_2) \in X \times \R$. Then
	\begin{align*}
		\langle x_1,x_2 \rangle = \langle y_1,y_2 \rangle + \rho_1 \rho_2 &= \langle y_1,y_2 \rangle + (-\rho_1) (-\rho_2)= \langle R_Ax_1,R_Ax_2 \rangle.
	\end{align*}
	
	\ref{p:Vsymmetry}: We have \eqref{p:Vsymmetry1} as an immediate consequence of \eqref{lem:V}, and \eqref{p:partialVsymmetry} as an immediate consequence of \eqref{V:subdiff}.
	
	\ref{p:Tequivalence}: This follows readily from the linearity and self-inverse properties of $R_A$. See, for example, \cite[Proposition~2.5 and Theorem~2.7]{HMorder}.
\end{proof}

\subsubsection*{A motivating example}

We next prove a theorem that establishes a relationship between subgradients of $V$ and circumcentered reflection method (CRM). Let us first explain a prototypical example of this relationship, before proving it formally. At left in Figure~\ref{fig:Lyapunovcircle}, we see that
\begin{align}
	&(\exists \mu_1,\mu_2,\mu_3 \in \R)\quad \text{such\; that} \label{Benoistprototype1}\\ 
	&\resizebox{1\hsize}{!}{%
		$C(x,R_Ax,R_BR_Ax)=P_Ax-\mu_1\nabla V(P_Ax) = P_BR_Ax - \mu_2 \nabla V(P_BR_Ax) = T_{A,B}x-\mu_3\nabla V(T_{A,B}x),$%
	}\nonumber
\end{align}
where $V$ is Benoist's Lyapunov function. Likewise, at right in Figure~\ref{fig:Lyapunovcircle}, we see that
\begin{align}
	&(\exists \mu_1,\mu_2,\mu_3 \in \R)\quad \text{such\; that} \label{Benoistprototype2}\\ 
	&\resizebox{1\hsize}{!}{%
		$C(x,R_Bx,R_AR_Bx)=P_Bx-\mu_1\nabla V(P_Bx) = P_AR_Bx - \mu_2 \nabla V(P_AR_Bx) = T_{B,A}x-\mu_3\nabla V(T_{B,A}x).$%
	}\nonumber
\end{align}
In other words, the CRM updates for the chosen point $x$ in both cases may be characterized as gradient descent applied to Benoist's Lyapunov function with the special step sizes determined by \eqref{Benoistprototype1} and \eqref{Benoistprototype2} for the two operators $T_{A,B}$ and $T_{B,A}$ respectively. 

In Theorem~\ref{thm:gradientdescent}, we show that a weaker version of this holds in higher dimensional Euclidean spaces for a broader class of problems covered by the framework of Dao and Tam. Specifically, we show that, under mild conditions, the \emph{perpendicular bisector} of any given side of the triangle $(x,R_Ax,R_BR_Ax)$ or the triangle $(x,R_Bx,R_AR_Bx)$ with midpoint $p$ contains $p+\R p^*:=\{p+\mu p^*\;|\; \mu \in \R\}$ for some $p^* \in \partial V(p)$.

\subsubsection*{The general case in Euclidean space}

\begin{theorem}\label{thm:gradientdescent}
	Let $A,B$ be as specified in \eqref{eqn:AB}. Let $f$ be Lipschitz continuous on $D$, and let $D,V,F$ be as in Lemma~\ref{lem:tangency}. Suppose further that 
	\begin{equation}\label{assumptionf0}
		(\forall y \in D)(0 \in \partial^0f(y)) \implies f(y)=0.
	\end{equation}
	Let $x \in D\times \R$ and $P_B$ be single-valued\footnote{This assumption is not strictly necessary, but the greater generality of forgoing it does not merit the burdensome notation that would be needed to do so.} on $D\times \R$ (combined with the single-valuedness $P_A$, this forces $T_{A,B},T_{B,A}$ to be single-valued, simplifying the notation). Then we have the following:
	\begin{enumerate}[label=(\roman*)]
		\item\label{TABH} There exists $\bar{x}^+\in \partial V(T_{A,B}x)$ such that
		\begin{equation*}
			(\forall \mu \in \R)\;\; T_{A,B}x-\mu \bar{x}^+ \in H(x,R_BR_Ax). 
		\end{equation*}
		\item\label{PAH} $\partial V(P_Ax) \neq \emptyset$ and
		\begin{equation}\label{e:PAH}
			(\forall \mu \in \R)(\forall \bar{x} \in \partial V(P_Ax)) \quad P_Ax - \mu\bar{x} \in H(x,R_Ax).
		\end{equation}
		\item\label{PBH} There exists $x^* \in \partial V(P_Bx)$ such that 
		\begin{equation}\label{e:PBH}
			(\forall \mu \in \R) \quad P_Bx - \mu x^* \in H(x,R_Bx).
		\end{equation}
		\item\label{TBAH}There exists $\bar{x}^* \in \partial V(T_{B,A}x)$ such that
		$$
		(\forall \mu \in \R)\;\; T_{B,A}x-\mu \bar{x}^+ \in H(x,R_AR_Bx). 
		$$
	\end{enumerate}
\end{theorem}
\begin{proof}
	\ref{TABH}: Because the conditions of Lemma~\ref{lem:tangency}\ref{lem:tangent_x+} are satisfied, we have that there exists $\bar{x}^+ \in \partial V(T_{A,B}x)$ that satisfies \eqref{V:subdiff2}, and so $\bar{x}^+ \in \{x-T_{A,B}x\}^\perp=\{x-R_B R_A x\}^\perp$. Using the fact that $\{x-R_B R_A x\}^\perp$ is a subspace and $T_{A,B}x \in H(x,R_B R_A x)$, we obtain that	
	\begin{equation*}
		(\forall \mu \in \R)\quad T_{A,B}x -\mu \bar{x}^+ \in H(x,R_B R_Ax).
	\end{equation*}
	This shows \ref{TABH}.
	
	\ref{PAH}: Because $x \in D \times \R$, we have that $x=(y,\rho)$ for some $y \in D$, and so $P_Ax = (y,0) \in D\times\{0\}$. We will consider two cases: $f(y)=0$ and $ f(y)\neq0$. 
	
	\textbf{Case 1:} Let $f(y)=0$. Combining $(f(y)=0)$ together with the fact that $y \in D$, we obtain  
	\begin{align}\label{PAH1}
		&\partial V(P_Ax) = \partial V(y,0) \stackrel{(a)}{=} \partial F(y)\times\{0\} \stackrel{(b)}{=} \{0\}_X \times\{0\}_{\R} \subset H(x,R_Ax).
	\end{align}
	Here (a) is from \eqref{V:subdiff} while (b) uses \eqref{lem:F} together with the fact that $f(y)=0$. Now, combining \eqref{PAH1} with the fact that $A = H(x,R_Ax)$ is a subspace and $P_A(x) \in A$, we have 
	\begin{equation*}
		(\forall \mu \in \R)\quad P_Ax-\mu \bar{x} = P_Ax \in H(x,R_Ax),
	\end{equation*}
	which shows \ref{PAH} in the case when $f(y)=0$, and consequently for all cases when $0 \in \partial f^0(y)$. 
	
	\textbf{Case 2:} Now suppose $f(y) \neq 0$. Then, by our assumption \eqref{assumptionf0}, we have that $0 \notin \partial^0f(y)$. Consequently, we have from \eqref{lem:F} and the Lipschitz continuity of $f$ that $\partial F(y) \neq \emptyset$. Applying Lemma~\ref{lem:tangency}\ref{lem:tangent_curve}, we have from \eqref{V:subdiff} that 
	\begin{align}\label{PAH2}
		&\partial V(P_A(x)) = \partial V(y,0) = \partial F(y)\times\{0\} \subset H(x,R_Ax)
	\end{align}
	Combining \eqref{PAH2} with the fact that $\partial F(y)\neq \emptyset$, we have that $\partial V(P_Ax)$ is nonempty. Again combining with the fact that $H(x,R_Ax)=A$ is a subspace and $P_Ax \in A$, we obtain \eqref{e:PAH}.
	
	\ref{PBH}: Because $x\in D \times \R$, set $x=(y,\rho)$ with $y\in D, \rho \in \R$. Since $P_Bx \in B = \gra f$, we have that $P_Bx=(q,f(q))$ for some $q \in X$. We will show that $q \in D$. First notice that $x'=(y,-\rho) \in D \times \R$ and that $P_Bx = P_BR_Ax'$, because $x=(y,\rho)=R_A(y,-\rho)=R_Ax'$. Then notice that we may use the linearity of $R_A$ to write
	\begin{align*}
		P_Bx-(0,\rho)= P_BR_Ax'-(0,\rho)&=P_BR_Ax'-\frac{1}{2}(y,\rho)+\frac{1}{2}(y,-\rho)\\
		&=P_BR_Ax' - \frac{1}{2}R_Ax'+\frac{1}{2}x'\\
		&=\frac{1}{2}(2P_BR_Ax')-\frac{1}{2}R_Ax'+\frac{1}{2}x'\\
		&=\frac{1}{2}R_BR_Ax'+\frac{1}{2}x'\\
		&=T_{A,B}x' \in D \times \R.
	\end{align*}
	which shows that $P_Bx \in D \times \R$, and so $q \in D$. We will consider two cases: $0 \in \partial^0f(q)$ and $0 \notin \partial^0f(q)$. 
	
	\textbf{Case 1:} Let $0 \in \partial^0f(q)$. Then we have by assumption that $f(q)=0$, and so from \eqref{lem:F} that $\{0\} \subset \partial F(q)$. Using this fact together with \eqref{V:subdiff}, we have that
	\begin{equation*}
		\partial V(q,f(q)) = \partial F(q)\times\{f(q)\} \ni (0,f(q))=(0,0)=0=:x^*.
	\end{equation*}
	Clearly $x^*=0 \in \{x-R_Bx\}^\perp$. Using this, together with the fact that $P_Bx \in H(x-R_Bx)^\perp$, we have that \eqref{e:PBH} holds.
	
	\textbf{Case 2:} Let $0 \notin \partial^0f(q)$. We have from \cite[Lemma~3.4]{DT} that the relationship $(q,f(q))=P_B(y,\rho)$ is characterized by the existence of $q^* \in \partial^0f(q)$ such that
	\begin{equation}\label{lem34}
		y=q+\left(f(q)-\rho \right)q^*.
	\end{equation}
	Because $0 \notin \partial^0f(q)$, we have from \eqref{lem:F} that
	\begin{equation}\label{PBH1}
		\partial F(q) \ni \frac{f(q)}{\|q^*\|^2}q^*.
	\end{equation}
	Using \eqref{PBH1} together with \eqref{V:subdiff}, we have that
	\begin{equation}\label{def:x*}
		\partial V(P_Bx)= \partial V(q,f(q)) = \partial F(q)\times\{f(q)\} \ni \left(\frac{f(q)}{\|q^*\|^2}q^*,f(q)\right)=:x^*.
	\end{equation}
	Thus we have that
	\begin{subequations}\label{PBHlong}
		\begin{align}
			\langle x-P_Bx,x^* \rangle_{X\times \R} &= \left\langle (y-q,\rho-f(q)),x^* \right\rangle_{X\times \R} \label{PBHlong1}\\
			&=\left \langle \left( \left(f(q)-\rho\right) q^*,\rho-f(q) \right),x^* \right\rangle_{X\times \R} \label{BHlong2}\\
			&=\left\langle \bigg( \left(f(q)-\rho\right) q^*,\rho-f(q) \bigg),\left(\frac{f(q)}{\|q^*\|^2}q^* ,f(q) \right) \right\rangle_{X\times \R} \label{PBHlong3}\\
			&=\left \langle \left(f(q)-\rho \right)q^*, \frac{f(q)}{\|q^*\|^2}q^*   \right \rangle_{X} + \left \langle \rho - f(q), f(q) \right \rangle_{\R} \label{PBHlong4}\\
			&=\bigg((f(q)-\rho)f(q)\frac{\|q^*\|^2}{\|q^*\|^2}\bigg) + \big((\rho-f(q))f(q)\big)\nonumber\\
			&=\big((f(q)-\rho)f(q)\big) + \big((\rho-f(q))f(q)\big)\nonumber\\ 
			&= 0.\nonumber
		\end{align}
	\end{subequations}
	Here \eqref{PBHlong1} is true from the definitions $x=(y,\rho)$ and $P_Bx=(q,f(q))$, \eqref{BHlong2} uses the equality \eqref{lem34}, \eqref{PBHlong3} uses the identity \eqref{def:x*}, \eqref{PBHlong4} splits the single dot product term on $X \times \R$ into a sum of two dot product terms on $X$ and $\R$, and what remains is linear algebra.
	
	Altogether, \eqref{PBHlong} shows that $x^* \in \{x-P_Bx\}^\perp = \{x-R_Bx \}^\perp$. Combining this with the fact that $\{x-R_Bx \}^\perp$ is a subspace and that $P_Bx \in H(x,R_Bx)$, we obtain \eqref{e:PBH}. This concludes the proof of \ref{PBH}.
	
	\ref{TBAH}: Let $x = (y,\rho) \in D \times \R$. We know from \ref{TABH} of this Theorem that 
	\begin{equation}\label{TBAH2}
		(\exists x^* \in \partial V(T_{A,B}(y,-\rho)))\quad \text{such\;that}\quad \langle x^*,(y,-\rho)-T_{A,B}(y,-\rho)\rangle =0.
	\end{equation}
	Consequently, we obtain that
	\begin{subequations}\label{TBAH3}
		\begin{align}
			0&=\langle x^*,(y,-\rho)-T_{A,B}(y,-\rho)\rangle\label{TBAH3a}\\
			&=\langle R_A x^*,R_A(y,-\rho)-R_AT_{A,B}(y,-\rho)\label{TBAH3b}\\
			&=\langle R_Ax^*,(y,\rho)-R_AT_{A,B}R_A(y,\rho)\rangle \nonumber\\
			&=\langle R_Ax^*,x-T_{B,A}x \rangle,\label{TBAH3c}
		\end{align}
	\end{subequations}
	where \eqref{TBAH3a} is true from \eqref{TBAH2}, \eqref{TBAH3b} follows from Proposition~\ref{prop:handy}\ref{p:indirectmotion}, and \eqref{TBAH3c} uses Proposition~\ref{prop:handy}\ref{p:Tequivalence}. We also have that 
	\begin{align}
		R_A \partial V(T_{A,B}(y,-\rho)) \stackrel{(a)}{=} \partial V(R_AT_{A,B} (y,-\rho)) =\partial V(R_AT_{A,B} R_A(y,\rho)) \stackrel{(b)}{=} \partial V(T_{B,A}x),\label{TBAH4}
	\end{align}
	where (a) uses \eqref{p:partialVsymmetry2} from Proposition~\ref{prop:handy}\ref{p:Vsymmetry} and (b) uses Proposition~\ref{prop:handy}\ref{p:Tequivalence}. Now combining \eqref{TBAH4} with the fact that $x^* \in \partial V(T_{A,B}(y,-\rho)$ from \eqref{TBAH2}, we have that
	\begin{equation}\label{TBAH5}
		R_Ax^* \in \partial V(T_{B,A}x).
	\end{equation}
	Combining \eqref{TBAH3} and \eqref{TBAH5}, we have that 
	\begin{equation}\label{TBAH6}
		\bar{x}^* := R_Ax^* \in \{x-T_{B,A}x\}^\perp = \{x-R_AR_Bx\}^\perp \quad \text{and}\quad \bar{x}^* \in \partial V(T_{B,A}x).
	\end{equation}
	Combining \eqref{TBAH6} with the fact that $\{x-R_AR_Bx\}^\perp$ is a subspace and $T_{A,B}x \in H(x,R_AR_Bx)$, we have that $\bar{x}^*$ satisfies \ref{TBAH}. This concludes the proof of \ref{TBAH}, completing the proof of the theorem.	
\end{proof}

Theorem~\ref{thm:gradientdescent} makes the following result on $\R^2$ easy to show.

\begin{corollary}\label{cor:gradient}
	Let $f:\R \rightarrow \R$ be continuous and differentiable, let  $A,B$ be as in \eqref{eqn:AB}, and let $D,V,F$ be as in Lemma~\ref{lem:tangency}. Suppose further that $(\forall y \in D)\;f'(y) \neq 0$. Let $x \in D\times \R$. Then the following hold.
	\begin{enumerate}[label=(\alph*)]
		\item\label{c:gradientAB} If $x,P_Ax,P_BR_Ax \notin A \cap B$ are not colinear, then \eqref{Benoistprototype1} holds.
		\item\label{c:gradientBA} If $x,P_Bx,P_AR_Bx \notin A \cap B$ are not colinear, then \eqref{Benoistprototype2} holds.
	\end{enumerate}
\end{corollary}

Recall the aforementioned results of Dizon, Hogan, and Lindstrom \cite{DHL2019} that guarantee quadratic convergence of CRM for many choices of $f$. Corollary~\ref{cor:gradient} may be seen as showing that the specific gradient descent method for $V$ that corresponds to CRM in $\R^2$ actually has quadratic rate of convergence for choices of $f$ covered by their results. Interestingly, connections with Newton--Raphson in $\R^2$ do not end there.

\subsection{Newton--Raphson method and subgradient descent on $f$ as gradient descent on $V$}

The following proposition shows that Newton--Raphson and subgradient projections method for $f:X \rightarrow \R$ may also be characterized as gradient descent on $V$ with step size $1$.

\begin{proposition}\label{prop:NewtonRaphson}Let $f:X \rightarrow \R$ be continuous, $A,B$ be as in \eqref{eqn:AB}, and $D,V,F$ be as in Lemma~\ref{lem:tangency}. Let $x=(y,0) \in D \times \R$ and $0\notin \partial^0f(y)$. Let $\Newt$ be the Newton--Raphson method. The following hold.
	\begin{enumerate}[label=(\roman*)]
		\item\label{p:NewtonRaphson} If $X = R$ and $f$ is differentiable, then 
		\begin{equation*}
			(\Newt(y),0) = x-\nabla V(x).
		\end{equation*}
		\item\label{p:subgradientdescent} Otherwise, let $y^* \in \partial^0f(y)$, and we have
		\begin{equation}\label{e:subgradientdescent}
			\left(y-\frac{f(y)}{\|y^*\|^2}y^*,0 \right) = x-x^*,\quad \text{for some}\;x^* \in \partial V(x).
		\end{equation}
	\end{enumerate}
\end{proposition}
\begin{proof}
	\ref{p:NewtonRaphson}: Simply notice that, by \eqref{lem:F} and \eqref{lem:V},
	\begin{equation*}
		(\Newt(y),0) = \left(y-\frac{f(y)}{f'(y)},0 \right) = \left(y-\nabla F(y),0 \right) = (y,0)-\left( \nabla F(y),0 \right) = x-\nabla V(x).
	\end{equation*}
	
	\ref{p:subgradientdescent}: Again by \eqref{lem:F} and \eqref{lem:V}, 
	\begin{equation*}
		x^* := \left(\frac{f(y)}{\|y^*\|^2}y^*,0 \right) \in \partial F(y) \times \{0\}=\partial V(y,0)=\partial V(x).
	\end{equation*}
	This choice of $x^*$ clearly satisfies the requirements of \eqref{e:subgradientdescent}, showing \ref{p:subgradientdescent}.
\end{proof}
For many choices of $f$, the explicit equivalence with Newton--Raphson method given by Proposition~\ref{prop:NewtonRaphson}\ref{p:NewtonRaphson} actually guarantees quadratic convergence of gradient descent on $V$ with step size $1$ for $x$ started in $A$. Altogether, we have shown that CRM on $\R^2$, Newton--Raphson on $\R$, and subgradient projection methods on $\R^n$ may all be characterized as gradient descent applied to Lyapunov functions constructed to describe the Douglas--Rachford method for many prototypical problems. More importantly, we have Theorem~\ref{thm:gradientdescent}, which relates the subgradients of $V$ to the perpendicular bisectors of the triangles that are the basis of CRM in Euclidean space more generally. 

\section{Spherical surrogates for Lyapunov Functions}\label{s:LT}

From now on, we assume $T,U,V$ together satisfy \ref{A1}. We are particularly interested in operators of the following form, which will be illustrated by Corollary~\ref{cor:TABinSL} and whose significance is made clear by Theorem~\ref{thm:quadratic}.
\begin{definition}
	Let $V$ be a smooth Lyapunov function with respect to $(\omega_1,\omega_2)$ on $U$ for the difference inclusion $x^+ \in Tx$. Let $\Omega_T,\Lambda,\psi:E \rightarrow E$ satisfy\footnote{How exactly one defines the \textit{otherwise} (colinear) case is of little practical importance for our analysis. By setting it to be $x^+$, one re-attempts a surrogate-minimizing step after the computation of one more update of $T$. By setting it to be $Tx^{+}$, one computes two more updates. One could choose either value, or something else.}
	\begin{align}\label{def:OmegaT}
		&\Omega_T: x \mapsto \begin{cases}
			C(x,2x^+ -x,\Lambda x), &\text{if}\;\; x,2x^+ -x,\Lambda x \text{\;are\;not\;colinear;}\\
			x^+ & \text{otherwise},
		\end{cases},\\
		&\psi x \in H \left(x, \Lambda x \right)\cap \aff\{x,2x^+-x,\Lambda x \},\quad \text{and}\label{d:psi}\\
		&x^+ +\R \nabla V(x^+)\subset H(x,2x^+ -x) \quad  \text{and} \quad \psi x+\R \nabla V(\psi x ) \subset H(\Lambda x ,x).\nonumber
	\end{align}
	Then we say $\Omega_T$ \emph{minimizes a circumcenter-defined spherical surrogate} for $V$, or $\Omega_T \in \mathbf{MCS}(V)$ for short.
\end{definition}

The following corollary is an immediate consequence of Theorem~\ref{thm:gradientdescent}.

\begin{corollary}\label{cor:TABinSL}
	Let \ref{A1} 
	hold \emph{and} $D,T_{A,B},T_{B,A}$ be as in Theorem~\ref{thm:gradientdescent}, with $U=D\times\R$, and $\mathcal{C}_{T_{A,B}}$ as defined in \eqref{def:circumcenteredDR}. Then
	\begin{align*}
		&\mathcal{C}_{T_{A,B}}\in \mathbf{MCS}(V)\;\;\text{with}\;\;\Lambda  =R_A\;\;\text{and}\;\;\psi =P_A,\\
		\text{and}\quad&\mathcal{C}_{T_{B,A}}\in \mathbf{MCS}(V)\;\;\text{with}\;\;\Lambda  =R_B\;\;\text{and}\;\;\psi  =P_B.
	\end{align*}
\end{corollary}

Next, in Theorem~\ref{thm:quadratic}, we show that an operator in $\mathbf{MCS}(V)$ may be characterized as \textit{returning the minimizer of a spherical surrogate} for the Lyapunov function $V$. Figure~\ref{fig:quadratic} illustrates this for $C_T$, and for the new operator $L_T$ that we will introduce, when $T$ is the Douglas--Rachford operator. Before proving it, let us first formalize the notion.
\begin{definition}\label{def:MSS}
	Let $n,m \in \N$ and $\Theta_T:E \rightarrow E$. Further suppose that there exist maps $\sigma_1,\dots,\sigma_m,\Gamma_1,\dots,\Gamma_{n}:E \rightarrow E$ such that whenever $\aff \{x,\Gamma_1x, \dots, \Gamma_{n}x\}$ has dimension $n$, the function $Q: u \mapsto d(u,\Theta_T x)^2$ satisfies 
	\begin{equation*}
		(\forall j \in \{1,\dots,m\})\quad P_{\aff \{x,\Gamma_1 x,\dots,\Gamma_{n}x \}- x}(\nabla V(\sigma_jx)) \in {\rm span}\{(\nabla Q)(\sigma_j x) \}.\\
	\end{equation*}
	Then we say that $\Theta_T$ \textit{minimizes a $n$-dimensional spherical surrogate for $V$, fitted at $m$-points}, or $\Theta_T \in \mathbf{MSS}(V)_m^n$ for short.
\end{definition}
Now we will show that $\mathbf{MCS}(V) \subset \mathbf{MSS}(V)$.
\begin{theorem}[$\mathbf{MCS}(V) \subset \mathbf{MSS}(V)$]\label{thm:quadratic}
	Let $\Omega_T \in \mathbf{MCS}(V)$, and let $Q:u \mapsto d(u,\Omega_T x)^2$. Then whenever $x,2x^+-x,\Lambda x$ are not colinear, the following hold:
	\begin{enumerate}[label=(\roman*)]
		\item\label{thm:quadratic1} $P_{\aff \{x,2x^+-x,\Lambda x \}-x}(\nabla V(x^+)) \in {\rm span}\{(\nabla Q)(x^+) \}$; and
		\item\label{thm:quadratic2} $P_{\aff \{x,2x^+-x,\Lambda x \}-x}(\nabla V(\psi x)) \in {\rm span}\{(\nabla Q)(\psi x) \}$.
	\end{enumerate}
	Accordingly, $\mathbf{MCS}(V) \subset \mathbf{MSS}(V)_2^2$ with $\Gamma_1:x\mapsto 2x^+-x$, $\Gamma_2:x \mapsto \Lambda x$, $\sigma_1:x \mapsto x^+$ and $\sigma_2:x \mapsto \psi x$.
\end{theorem}
\begin{proof}
	Fix $x$. We can handle both cases \ref{thm:quadratic1} and \ref{thm:quadratic2} with the same argument: by letting $(a,b,c):=(x,2x^+-x,x^+)$ in the former case or $(a,b,c):=(x,\Lambda x,\psi x)$ in the latter. Let $x,2x^+-x,\Lambda x$ be not colinear. Then $\Omega_T x=C(x,2x^+-x,\Lambda x)$ exists and so satisfies
	\begin{align*}
		\Omega_T x &\in C_{ab} \quad \text{where}\quad C_{ab}:=H(a,b) \cap \aff\{x,2x^+-x,\Lambda x \}
	\end{align*}
	is an affine subspace of dimension $1$. Next, notice that
	\begin{align*}
		\aff\{x,2x^+-x,\Lambda x \} -x &\stackrel{(a)}{=} \aff\{x,2x^+-x,\Lambda x \} -x - (\Omega_Tx-x)\\
		&=\aff\{x,2x^+-x,\Lambda x \} - \Omega_Tx.
	\end{align*}
	Here (a) holds because $\Omega_Tx \in \aff\{x,2x^+-x,\Lambda x \}$, and so $\Omega_Tx-x \in \aff\{x,2x^+-x,\Lambda x \} -x$, where the latter is a subspace and therefore invariant under translation by any of its members (including, specifically $\Omega_Tx -x$). The problem thus simplifies to showing
	\begin{equation*}
		P_{\aff\{x,2x^+-x,\Lambda x \}-\Omega_Tx}\left(\nabla V(c) \right) \in {\rm span}\{\nabla Q(c) \}.
	\end{equation*}
	By a suitable translation and with no loss of generality, we may let $\Omega_T x = 0$, which simplifies our notation and allows the problem to be written as:
	\begin{equation*}
		P_{\aff\{x,2x^+-x,\Lambda x \}}\left(\nabla V(c) \right) \in {\rm span}\{\nabla Q(c) \}.
	\end{equation*}
	Moreover, because $\Omega_Tx=0$, we have that $C_{ab}$, $H(a,b)$, and $\aff\{x,2x^+-x,\Lambda x \}$ are all subspaces. Since $H(a,b)$ is a subspace, it is invariant under translation by any of its members; in particular we have all the equalities:
	\begin{equation}\label{quadratic1}
		H(a,b)=H(a,b)-c=H(a,b)-(a+b)/2=\{a-b\}^\perp.
	\end{equation}
	Furthermore, because $\Omega_T \in \mathbf{MCS}(V)$, we have $c+\R \nabla V(c) \subset H(a,b)$, and so
	\begin{equation}\label{quadratic2}
		\nabla V(c) \in H(a,b)-c = \{a-b\}^\perp,
	\end{equation}
	where the equality is from \eqref{quadratic1}. From \eqref{quadratic2}, we may write
	\begin{align*}
		\nabla V(c) &= u+v \quad \text{where}\\
		u \in \aff\{x,2x^+-x,\Lambda x \}\cap \{a-b\}^\perp=C_{ab} \quad &\text{and}\quad v \in  \aff\{x,2x^+-x,\Lambda x \}^\perp\cap \{a-b\}^\perp.
	\end{align*}
	Consequently, we have that
	\begin{align}\label{quadratic4}
		P_{\aff\{x,2x^+-x,\Lambda x \}}\left(\nabla V(c) \right) &= P_{\aff\{x,2x^+-x,\Lambda x \}}(u+v) = u \in C_{ab}.
	\end{align}
	Now since $\Omega_T x = 0$, we have $Q = \|\cdot \|^2$, and so $\nabla Q(w)=2w$ for all $w$. In particular, $\nabla Q(c)=2c$ where $c \neq \Omega_T x = 0$. From the definition of $\Omega_T$, it is clear that $c \in C_{ab}$. Altogether,
	\begin{equation*}
		P_{\aff\{x,2x^+-x,\Lambda x \}}\left(\nabla V(c) \right) \stackrel{(i)}{\in} C_{ab} \stackrel{(j)}{=} {\rm span}\{c \} \stackrel{(k)}{=} {\rm span}\{\nabla Q(c) \}.
	\end{equation*}
	Here (i) is true from \eqref{quadratic4}, (j) holds because $c\neq0$ and $c \in C_{ab}$ where $C_{ab}$ is a subspace of dimension 1, and (k) holds because $\nabla Q(c)=2c$. This shows the result.
\end{proof}

\begin{figure}
	\begin{center}
		\includegraphics[width=.29\textwidth]{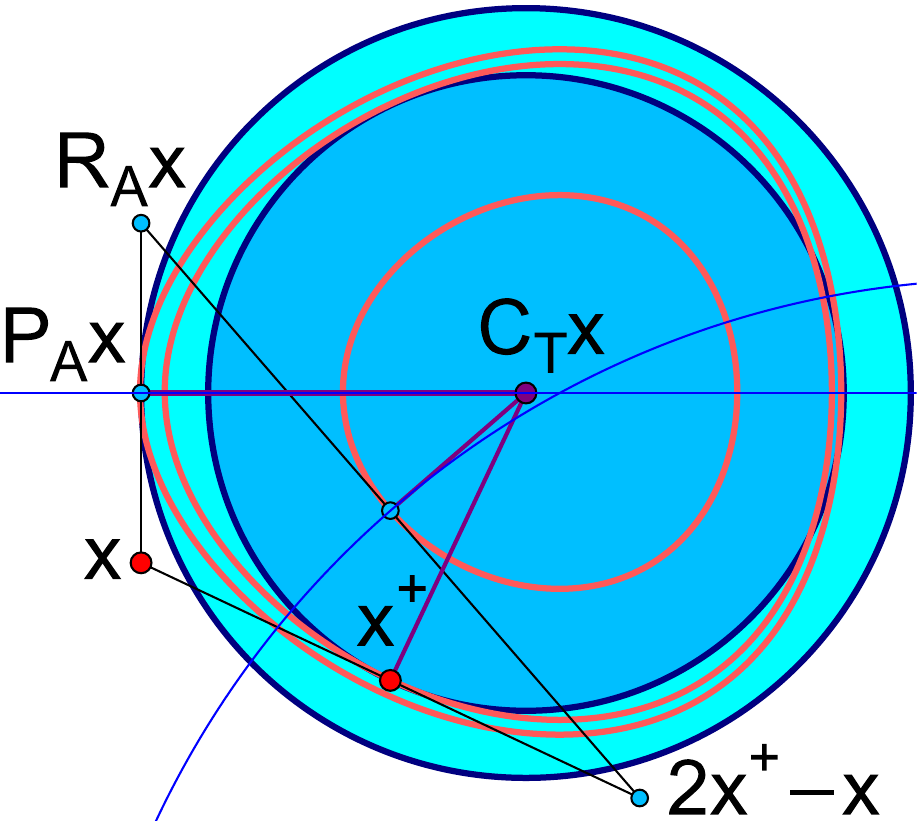}\;\includegraphics[width=.27\textwidth]{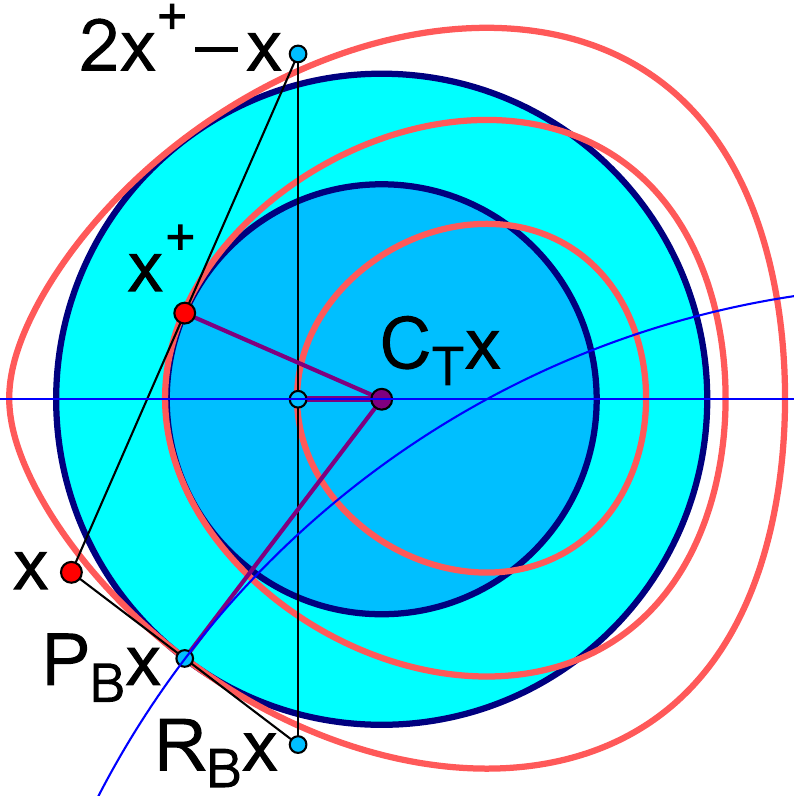}\;\includegraphics[width=.31\textwidth]{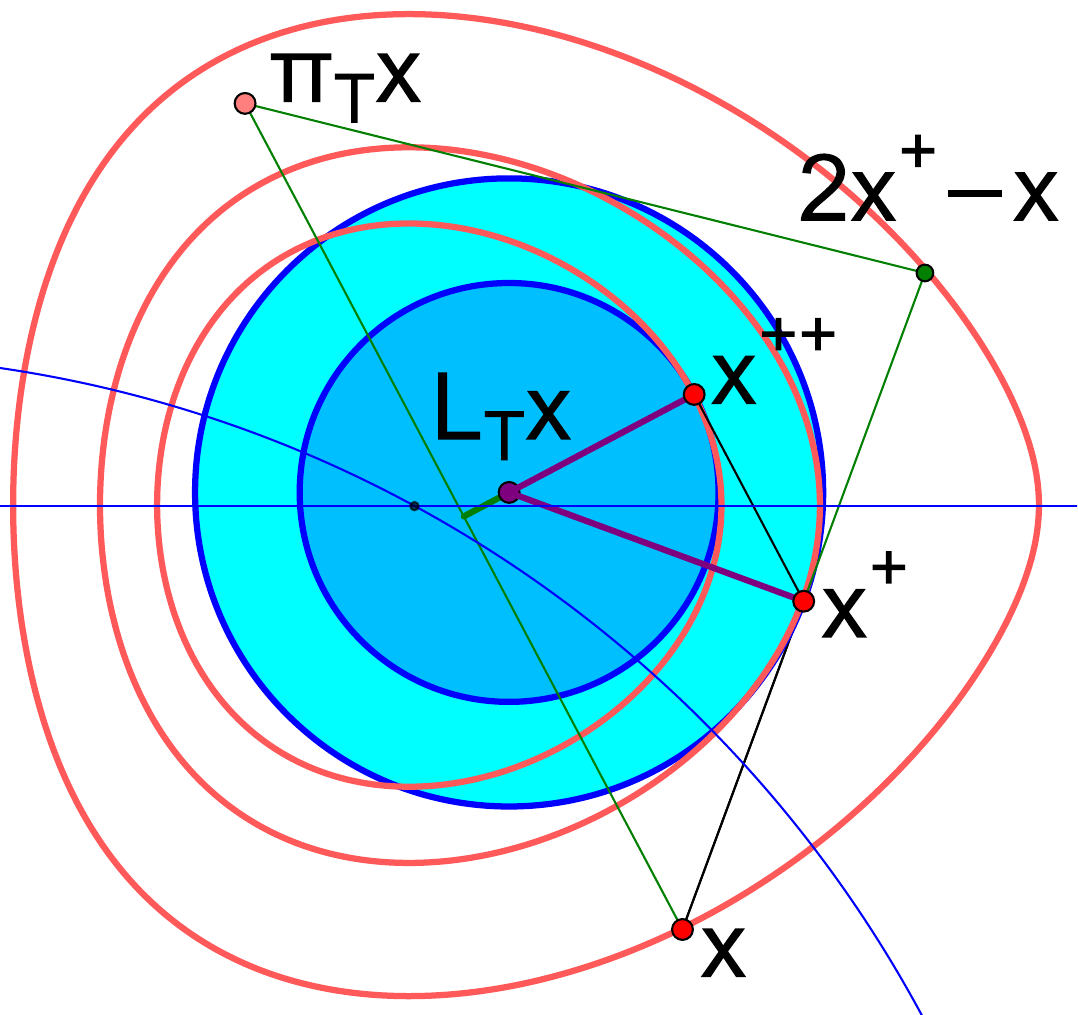}
	\end{center}
	\caption{When $V$ is Benoist's Lyapunov function, $C_{T_{A,B}},C_{T_{B,A}},L_{T_{A,B}} \in \mathbf{MCS}(V)$ (respectively from left to right) and so they minimize spherical surrogates as described in Theorem~\ref{thm:quadratic}.}\label{fig:quadratic}
\end{figure}

Now we will introduce a new operator in $\mathbf{MCS}(V)$ that has the additional property that it is defined for a general operator $T$ and \emph{does not depend} on substeps (e.g. reflections). This is highly advantageous. For example, this property allows the new operator to be used for the basis pursuit problem in Section~\ref{s:duality}, where $T_{\A,\B}$ is as in \eqref{LM}, the problem \eqref{eqn:operator_sum_problem} is more general than \eqref{feasibility_problem}, and the proximity operator $\A:=\prox_{cd_2}$ is no longer a projection.

\begin{definition}\label{def:LT}
	Denote $x^+\in Tx$ and $x^{++}\in Tx^+$. Let $L_T,\pi_T$ be as follows:
	\begin{align*}
		\pi_T:E \rightarrow E: x\mapsto&2(x^{++}-x^{+})+2P_{{\rm span}(x^{++}-x^{+})}(x^{+}-x)+x,\\
		&=2(x^{++}-x^{+})+2\left(\frac{\langle x^{+}-x,x^{++}-x^+ \rangle}{\|x^{++}-x^+\|^2}(x^{++}-x^+) \right)+x, \\
		\text{and}\quad L_T:E \rightarrow& E: x \mapsto \begin{cases}
			C(x,2x^+ -x,\pi_T x), &\text{if}\;\; x,2x^+ -x,\pi_T x \text{\;\;are\;not\;colinear;}\\
			x^{+} & \text{otherwise}
		\end{cases}
	\end{align*}
	The construction of the operator $L_T$ is principally motivated by minimizing the spherical surrogate of a Lyapunov function for $T$.
\end{definition}
One step of $L_{T_{B,A}}$ is shown for Example~\ref{ex:2lines} in Figure~\ref{fig:lines_and_quadratic} (left), together with reflection substeps. In Figure~\ref{fig:lines_and_quadratic} (right), where $B$ is a circle and $A$ a line, we omit the reflections in order to highlight that the construction of $L_T$ for a general operator $T$ depends only on $(x,x^+,x^{++})$. Now we have the main result about $L_T$.
\begin{theorem}\label{thm:LT}
	Let \ref{A1} 
	hold. Then $L_T \in \mathbf{MCS}(V)$ with $\Lambda = \pi_T$ and $\psi = T^2$.
\end{theorem}
\begin{proof}
	Let, $\psi x=x^{++}$ and $\Lambda x=\pi_T x$ and let $x,2x^+-x,\pi_Tx$ be not colinear. Then we need only show that
	\begin{subequations}
		\begin{align}
			&x^+ -\R\nabla V(x^+) \subset H(x,2x^{+}-x),\label{tLtfirst} \\
			\text{and}\quad &x^{++} -\R\nabla V(x^{++}) \subset H(x,\pi_T x)\label{tLtsecond}
		\end{align}
	\end{subequations}
	The first inclusion \eqref{tLtfirst} is a straightforward consequence of \ref{A1}, and so also is
	\begin{equation*}
		x^{++} - \R\nabla(x^{++}) \subset H(2x^{++}-x^+,x^+).
	\end{equation*}
	Thus we may show \eqref{tLtsecond} by showing
	\begin{equation}\label{LT0}
		H(2x^{++}-x^+,x^+) \subset H(\pi_Tx,x),
	\end{equation}
	which is what we now do. Because $\pi_Tx-x \in {\rm span}(x^{++}-x^+)$, we have that
	\begin{equation}\label{LT1}
		\{x^{++}-x^+\}^\perp {\subset} \{\pi_Tx-x \}^\perp.
	\end{equation}
	Now, for simplicity, set
	\begin{equation}\label{d:zy}
		z:=P_{{\rm span}(x^{++}-x^{+})}(x^{+}-x) \quad \text{and}\quad y := (x^+-x)-z \in \{x^{++}-x^+ \}^\perp.
	\end{equation}
	Then we have that 
	\begin{subequations}\label{LT3}
		\begin{align}
			2\left(\frac{\pi_T x+x}{2}- x^{++}\right) &=\pi_Tx+x-2x^{++}\nonumber \\
			&= \left( 2(x^{++}-x^+ )+2z+x \right)+x-2x^{++} \label{LT3a}\\
			&=-2x^+ +2z+2x  \nonumber \\
			&=-2(x+z+y) +2z+2x  \label{LT3b}\\
			&=-2y \in \{x^{++}-x^+ \}^\perp. \label{LT3c}
		\end{align}
	\end{subequations}
	Here \eqref{LT3a} uses the definition of $\pi_T$ together with the simplified notation from \eqref{d:zy}, \eqref{LT3b} substitutes for $x^+$ using \eqref{d:zy}, and the inclusion in \eqref{LT3c} is true from the definition of $y$ in \eqref{d:zy}. Altogether, we have
	\begin{subequations}
		\begin{align}
			\{x^{++}-x^+ \}^\perp - \frac{\pi_T x+x}{2}+ x^{++} &=\{x^{++}-x^+ \}^\perp  \label{LT4a}\\
			&\subset \{\pi_T x-x \}^\perp \label{LT4b}\\
			\text{and\; so}\quad \{x^{++}-x^+ \}^\perp + x^{++} &\subset \{\pi_T x-x \}^\perp +\frac{\pi_T x+x}{2}.\label{LT4c}
		\end{align}
	\end{subequations}
	Here \eqref{LT4a} applies \eqref{LT3}, \eqref{LT4b} uses \eqref{LT1}, and \eqref{LT4c} shows \eqref{LT0}, completing the result.
\end{proof}

In view of Theorem~\ref{thm:quadratic} and Theorem~\ref{thm:LT}, $L_T \in \mathbf{MSS}(V)_2^2$ with $\Gamma_1,\sigma_{1}:x \mapsto x^+$ and $\Gamma_2,\sigma_2:x \mapsto x^{++}$.

\subsection{Additional properties of $L_T$}

One of DR's advantageous qualities is thought to be that it often searches in a subspace of reduced dimension; for example, it solves the feasibility problem of two lines in $E$ by searching within a subspace of dimension $2$. The following proposition shows that $L_T$ maps spaces of reduced dimension into themselves whenever $T$ does.
\begin{proposition}\label{prop:dimension}
	The following hold.
	\begin{enumerate}[label=(\roman*)]
		\item\label{reduceddimensiona} $L_Tx \in \begin{cases}
			{\rm aff}\{x,x^+,x^{++}\} & \text{if} \;\;x,x^+,x^{++}\;\;\text{are\;not\;colinear;}\\
			Tx^{++} &\text{otherwise.}\end{cases}$
		\item\label{reduceddimensionb} If $U$ is an affine subspace and $T(U) \subset U$, then $L_T (U) \subset U$.
	\end{enumerate}
\end{proposition}
\begin{proof}
	Here \ref{reduceddimensiona} follows from the definition of $L_T$, and \ref{reduceddimensionb} follows from \ref{reduceddimensiona}.
\end{proof}
As Figure~\ref{fig:lines_and_quadratic} (left) would suggest, Proposition~\ref{prop:dimension} makes it straightforward to prove the following.
\begin{proposition}\label{prop:twolines}
	Let $A,B$ be lines in $E$. Then for any $x \in E$, $L_{T_{A,B}}x \in \Fix T_{A,B}$.
\end{proposition}
\begin{proof}
	For this problem, the Douglas--Rachford sequence converges in a subspace $U$ of reduced dimension $2$ and is equivalent to the problem for two lines in $U$. Therefore, by a suitable translation and without loss of generality, we may reduce to considering the problem in $\R^2$ from Example~\ref{ex:2lines}. If the two lines are perpendicular, $x^+ \in \Fix T_{A,B}$ and so $L_T x = x^{+} \in \Fix T_{A,B}$. If the two lines are not perpendicular, the result follows from Theorem~\ref{thm:LT} and the fact that the Lyapunov function on the subspace of reduced dimension $2$ is simply the spherical function in Example~\ref{ex:2lines} whose gradient descent trajectories all intersect only in $\Fix T_{A,B}$.
\end{proof}
Another characterization of Proposition~\ref{prop:twolines} is that, for two lines, the spherical surrogate $Q$ constructed by $L_{T_{A,B}}$, as described in Theorem~\ref{thm:quadratic}, is equal (up to rescaling) to the Lyapunov function $V$ for the Douglas--Rachford operator. Proposition~\ref{prop:twolines} highlights another difference between $L_T$ and $C_T$, because CRM may not converge finitely for this same problem \cite[Corollary~2.11]{circumcentering}. Other known results in the literature may also be easily proven via this explicit connection with Lyapunov functions, including results about CRM (e.g. \cite[Lemma~2]{behling2019convex}). 

Of course, it should be noted that CRM sometimes also converges in lower dimensional subspaces, as in Figure~\ref{fig:Lyapunovcircle} where CRM converges within the set $B$. This invariance was exploited in  \cite{DHL2019} and \cite{behling2019convex}, as described in Section~\ref{ss:circumcentering_introduction}. By contrast, $L_T$ does not converge within $B$, which highlights another difference between the methods.

However, the most important advantage of $L_T$ is that it does not depend on the substeps involved in computing $Tx$ from $x$ (e.g. reflections). This makes it a potential candidate for algorithms that \textit{show signs of spiraling} admitted by \emph{any} operator $T$, wherefore one might \textit{suspect} that the Lyapunov function satisfies the spiraling condition \ref{A1}
. The inclusion $C_T \in \mathbf{MCS}(V)$ from Corollary~\ref{cor:TABinSL} uses additional assumptions on the structure of $T_{A,B}$ that may not be satisfied for the more general operator in \eqref{LM}. In fact, in Section~\ref{s:duality}, we will actually show that CRM's dependence on the subproblems renders it useless for the basis pursuit problem, even though the iterates generated by $T_{A,B}$ \textit{show signs of spiraling}. On the other hand, $L_T \in \mathbf{MCS}(V)$ \emph{whenever} \ref{A1} 
is satisfied, and $L_T$ shows very promising performance for the basis pursuit problem.

\section{Primal/Dual Implementation}\label{s:duality}

In Section~\ref{s:LT} we introduced the operator $L_T$, whose inclusion in $\mathbf{MCS}(V)$ we showed with very few assumptions about the specific problem and operator structure. In this section, we describe how one may use a method from $\mathbf{MSS}(V)$ for the \textit{general} optimization problem \eqref{objective} by exploiting a duality relationship and demonstrating with $L_T$. The basic strategy is to reconstruct the spiraling dual iterates from their primal counterparts, and then to apply a surrogate-minimizing step. One then obtains a multiplier update candidate from the shadow of the minimizer for the surrogate, propagates this update back to the primal variables insofar as is practical, and then can compare this candidate against a regular update before returning to primal iteration.

We illustrate, in particular, with ADMM, which solves the augmented Lagrangian system associated with \eqref{objective} where $E=\R^n$ and $Y=\R^m$ via the iterated process
\begin{subequations}\label{ADMM}
	\begin{alignat}{1}
		x_{k+1} &\in \underset{x \in \R^n}{\argmin} \left \{f(x)+g(z_k)+ \langle \lambda_k,Mx-z_k \rangle + \frac{c}{2}\|Mx-z_k\|^2  \right \}\label{ADMMx} \\
		z_{k+1} &\in \underset{z \in \R^m}{\argmin} \left \{f(x_{k+1})+g(z)+ \langle \lambda_k,Mx_{k+1}-z \rangle + \frac{c}{2}\|Mx_{k+1}-z\|^2 \right \}\label{ADMMz} \\
		\lambda_{k+1}&=\lambda_k+c(Mx_{k+1}-z_{k+1}).\label{ADMMlambda}
	\end{alignat}
\end{subequations}
When $f,g$ are convex, ADMM is dual to DR for solving the associated problem:
\begin{align*}
	\underset{\lambda \in \R^m}{\rm minimize}\quad d_1(\lambda)+d_2(\lambda) \quad {\rm where}\quad d_1:=f^*\circ (-M^T)\quad \text{and}\quad d_2:=g^*.
\end{align*}
Here $f^*,g^*$ denote the Fenchel--Moreau conjugates of $f$ and $g$. For brevity, we state only how to recover the dual updates from the primal ones; for a detailed explanation, we refer the reader to the works of Eckstein and Yao \cite{eckstein2012augmented,eckstein2015understanding}, whose notation we closely follow, and also to Gabay's early book chapter \cite{Gabay}, and to the references in \cite{LSsurvey}. For strong duality and attainment conditions, see, for example, \cite[Theorem 3.3.5]{BL}. For a broader introduction to Langrangian duality, see, for example, \cite{RW98,bertsekas2009convex}. The dual (DR) updates $(y_k)_{k \in \N}$ may be computed from the primal (ADMM) thusly:
\begin{equation*}
	\begin{tabular}{c|c}
		primal & dual \\
		{$\!\begin{aligned}
				cMx_{k+1} &= \prox_{cd_1}(R_{cd_2}y_k)-R_{cd_2}(y_k) \\  
				cz_k &= y_k-\prox_{cd_2}y_k \\
				\lambda_k&=\prox_{cd_2}y_k \end{aligned}$}& {$\!\begin{aligned}
				y_{k}&= \lambda_k+cz_k \\  
				R_{cd_2}y_k&=\lambda_k-cz_k\\
				R_{cd_1}R_{cd_2}y_k&=\lambda_k-cz_k+2cMx_{k+1} \end{aligned}$}
	\end{tabular}
\end{equation*}
Here the reflected resolvents 
\begin{equation*}
	R_{cd_j} = 2J_{c\partial d_j}-\Id = 2\prox_{cd_j}-\Id \quad (j=1,2),
\end{equation*}
are the reflected proximity operators for $d_1,d_2$ \eqref{def:prox}. They are denoted by $N_{cd_1},N_{cd_2}$ in \cite{eckstein2012augmented,eckstein2015understanding}. The sequence of multipliers for ADMM corresponds to what is frequently called the ``shadow'' sequence for DR: $(\lambda_k)_{k \in \N} = (\prox_{cd_2}y_k)_{k \in \N}$. The difference of subsequent iterates thereof, $\|\lambda_{k+1}-\lambda_k\|$, is $\|P_Ax-P_Ax^+\|$ in Figure~\ref{fig:ellipse2}(right) for $d_2 = \iota_{A}$ where $A$ is the line in Figure~\ref{fig:ellipse2}(left). For feasibility problems \eqref{feasibility_problem}, the visible shadow oscillations have been consistently associated with \textit{showing signs of spiraling} observed in Figure~\ref{fig:ellipse2}(left). This suggests that primal problems eliciting such multiplier update oscillations---whereby we \textit{suspect} that the dual sequence $(y_k)_{k \in \N} \subset \R^m$ may be spiraling in the sense of \ref{A1}---are natural candidates for primal/dual $\mathbf{MSS}(V)$ methods. For example, one may consider applying
\begin{align}
	\mathbf{y}_{L_T}\leftarrow L_{T_{\partial d_2, \partial d_1}}(y_k) &=\begin{cases}
		C(y_k,2y_{k+1}-y_k,\pi_{T_{\partial d_2, \partial d_1}} y_k)&\\
		\text{if}\;y_k,y_{k+1},y_{k+2} \;   \text{are\;not\;colinear}.
	\end{cases};\label{LTdual}\\
	\text{or}\quad\mathbf{y}_{C_T}\leftarrow \mathcal{C}_{T_{\partial d_2, \partial d_1}}(y_k) &= \begin{cases}
		C(y_k,R_{cd_2}y_k,R_{cd_1}R_{cd_2}y_k) & \\ \text{if}\;y_k,R_{cd_2}y_k,R_{cd_1}R_{cd_2}y_k \;   \text{are\;not\;colinear}.
	\end{cases} \label{CTdual}
\end{align}
The former \eqref{LTdual} is the $L_T$ method associated to the Douglas--Rachford operator $T_{\partial d_2, \partial d_1}$ as described in \eqref{LM} for the maximal monotone operators $\partial d_2$ and $\partial d_1$. The latter \eqref{CTdual} may be seen as a generalization of the circumcentered reflection method that uses reflected proximity operator substeps in place of reflected projections. Remember that this second method may \textit{not} be in $\mathbf{MSS}(V)$ for this more general optimization problem, \textit{even if} \ref{A1} holds; in fact, we will see its failure for the basis pursuit problem. 

Once the update $\mathbf{y}_{\Omega_T}$ is computed, its shadow---$\prox_{cd_2}\mathbf{y}_{\Omega_T} \in \R^m$---is a candidate for the updated multiplier $\lambda^+$. One may evaluate the objective function in order to decide whether to accept it or reject it in favor of a regular multiplier update. Naturally, in the case when the components are colinear, one would proceed with a regular update.

\subsection{Example: Basis Pursuit}

\begin{figure}
	\begin{center}
		\includegraphics[width=.85\textwidth]{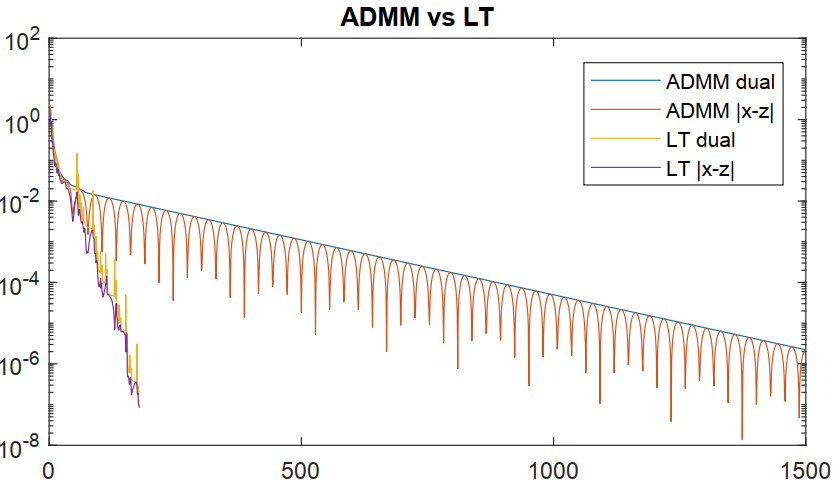}
	\end{center}
	\caption{$L_T$ (with objective function check) vs regular ADMM for the basis pursuit problem.}\label{fig:basis_pursuit}
\end{figure}

We will apply the algorithm to the \emph{basis pursuit} problem, for which the dual (DR) frequently \textit{shows signs of spiraling}, as in Figures~\ref{fig:basis_pursuit} and \ref{Basis_Pursuit_cindy}. Of course, a \textit{complete theoretical investigation} of this \textit{specific problem}---including the construction of the associated Lyapunov function for the dual---would likely constitute \textit{multiple} further articles. In the conclusion, we will suggest those projects as a natural further step of investigation; of course, such extensive and structure-specific work is beyond the scope and purpose of the present work. We include this experiment to demonstrate the \textit{implementation} and \textit{success} of our new algorithm for a primal/dual implementation, and not to make any absolute theoretical claims about convergence or rates. The basis pursuit problem,
\begin{equation*}
	\text{minimize}\quad \|x\|_1 \quad \text{subject\;to}\quad Ax=b,\quad x \in \R^n,\;A\in \R^{\nu \times n},\; b \in \R^\nu,\;\nu<n,
\end{equation*}
may be tackled by ADMM \eqref{ADMM} via the reformulation:
\begin{equation*}
	f:=\iota_{S},\;\;S:=\left \{x \in \R^n\;|\;Ax=b \right \}, \;\;M:= \Id,\;\;g:z \rightarrow \|z\|_1,\;\;E,Y:=\R^n.
\end{equation*}
The first update \eqref{ADMMx} is given by $x_{k+1}:=P_S(z_k-\lambda_k)$, and the second \eqref{ADMMz} by $z_{k+1}:={\rm Shrinkage}_{1/c}(x_{k+1}+\lambda_k)$. They may be computed efficiently; see the work of Boyd, Parikh, Chu, Peleato, and Eckstein \cite[Section~6.2]{boyd2011distributed}. We also have that
\begin{equation*}
	d_2 = \|\cdot\|_1^* = \iota_{\mathcal{B}_\infty}\quad \text{where}\; \mathcal{B}_\infty:=\{x\;|\; \|x\|_\infty \leq 1\},\quad \text{and\;so}\quad \prox_{cd2}=P_{\mathcal{B}_\infty},
\end{equation*}
is computable. After computing three updates of the dual (DR) sequence, $(y_k, y_{k+1}, y_{k+2})$, we update the DR sequence by using $L_T$ as in \eqref{LTdual}. Our multiplier update candidate is then $\lambda_{L_T} =P_{\mathcal{B}_\infty}\mathbf{y}_{L_T}$, and we propagate this update to the second variable \eqref{ADMMz} by $z_{L_T}=\mathbf{y}_{L_T}-\lambda_{L_T}$. We assess these against the regular update candidates by comparing their resultant objective function values,
\begin{equation*}
	\|x_{L_T}\|_1= \|P_S(z_{L_T}-\lambda_{L_T})\|_1 \quad \text{and}\quad \|x_{\text{REGULAR}}\|_1=\|P_S(z_{k+2}-\lambda_{k+2})\|_1,
\end{equation*}
and updating $\lambda_{k+2},x_{k+3}$ to match the winning candidate. 

\begin{figure}
	\begin{center}
		\includegraphics[width=.95\textwidth]{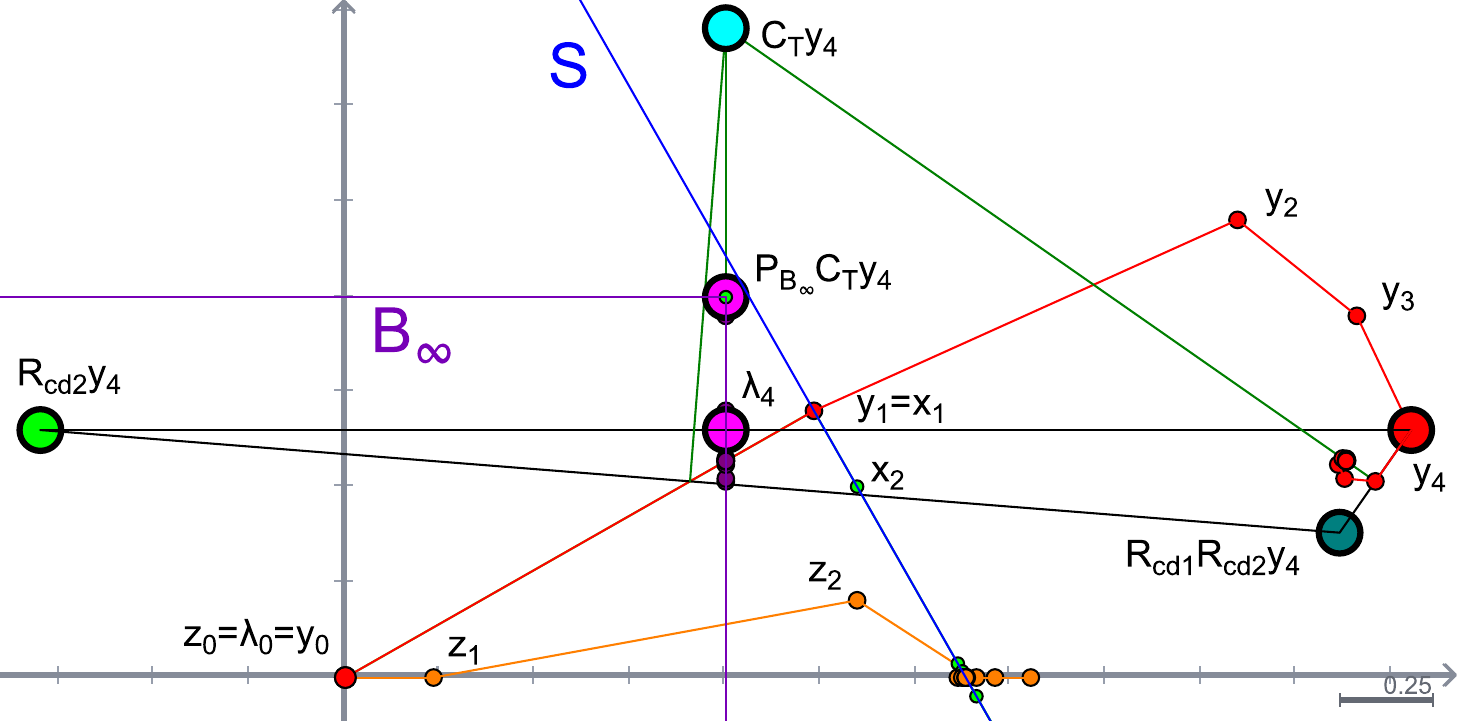}
	\end{center}
	\caption{The failure of $C_T$ for the basis pursuit problem.}\label{Basis_Pursuit_cindy}
\end{figure}

Figure~\ref{fig:basis_pursuit} shows the $L_T$ primal/dual approach together with vanilla ADMM for comparison. This juxtaposition of $L_T$ primal/dual method with vanilla ADMM resembles what has already been observed with CRM and Douglas--Rachford for nonconvex feasibility problems \cite{DHL2019}. The problem used was a randomly\footnote{Matlab code rand('seed', k); randn('seed', k); for k=1..1000} generated instance with $\text{seed}=0,n=30,\nu=10,c=1$, and the horizontal axis reports the number of passes through \eqref{ADMM}. This is the example problem from Boyd, Parikh, Chu, Peleato, and Eckstein's ADMM code, available at \cite{BoydADMMcode}. Our Matlab code is a modified version of theirs, and it is available at \cite{LindstromADMMcode}, together with the Cindrella scripts used to produce the other images in this paper. For 1,000 similar problems with $c=1$ and ``solve'' criterion 
\begin{align*}
	\|z_k-x_k\|<10^{-8}(\sqrt{n}+\max\{\|x_k\|,\|z_k\| \}) \quad \text{and}\quad \|z_k-z_{k-1}\| < 10^{-8}(\sqrt{n}+\|\lambda_k\|),
\end{align*}
$L_T$ and vanilla ADMM solved all problems and performed as in the following table. The column ``wins'' reports the number of problems (out of 1,000) for which a given method was faster than the other. Where an algorithm requires $n_k$ iterates to solve problem $k$, the statistics min, max, and quartiles Q1, Q2 (median), and Q3 describe the distribution of the set $\{n_1,\dots,n_{1,000}\}$.
\begin{equation*}
	\begin{tabular}{l| r r r r r r }
		& wins & min &  Q1 & Q2 & Q3 & max \\ \hline
		vanilla ADMM & 1 & 278 & 995 & 1582 & 2932 & 761,282 \\
		$L_T$ & 999 & 87 & 165 & 221 & 324 & 94,591
	\end{tabular}
\end{equation*}
The performance in Figure~\ref{fig:basis_pursuit} is typical of what we observed. Attempts to update the dual by using $C_T$ as in \eqref{CTdual} yielded an algorithm that consistently failed to solve the problem. Figure~\ref{Basis_Pursuit_cindy} shows why we would \emph{not expect} $C_T$ to work. For $y_k$ among the spiraling DR iterates $(y_k)_{k \geq 4}$ on the right, $H(y_k,R_{\mathcal{B}_\infty}y_k)$ is the vertical line containing the right side of the unit box. The CRM update $C_Ty_k$ will lie in this line, and so we would not expect it to be anywhere near the spiraling DR iterates, nor would we expect $P_{\mathcal{B}_\infty}{C_T}y_k$ to closely approximate the limit of the dual shadow (multiplier) sequence $(\lambda_k)_{k \in \N}$. In the figure, we have plotted and enlarged these points of the construction with $k=4$. In contradistinction, the operator $L_T$ depends only on the governing DR sequence, and so it is immune to this problem.

\section{Conclusion}\label{s:conclusion}

The ubiquity of Lyapunov functions (as recalled in Section~\ref{s:stability}), and the tendency of algorithms that show \textit{signs of spiraling} to satisfy \ref{A1} (as in \cite{Benoist,DT,giladi2019lyapunov}) suggest that future works should consider Lyapunov surrogate methods for other general optimization problems of form \eqref{objective} via the duality framework we introduced in Section~\ref{s:duality}. We have shown how such an extension is made broadly possible with generically computable operators selected from $\mathbf{MSS}(V)$. Moreover, we have already constructed one example, $L_T$, and experimentally demonstrated its success for a primal/dual adaptation. 

We should emphasize that $L_T$ is only one natural example from the much broader class of surrogate-motivated algorithms, and we have only focused on it to draw attention to the broader reasons for studying $\mathbf{MSS}(V)$. Naturally, one need not be limited to algorithms in $\mathbf{MCS}(V)$, which minimize spherical surrogates on affine subspaces of dimension $2$. One could easily work with spherical surrogates built from more iterates and higher dimensional affine subspaces. In fact, one could generalize even further, from minimizing \textit{spherical} surrogates to minimizing \textit{ellipsoidal} ones, or even \textit{general quadratic} ones. The specific example $L_T$ is only the beginning.

Convergence results for nonconvex problems are generally more challenging than for convex ones, and Lyapunov functions have already played an important role for the understanding of nonconvex DR. The Lyapunov function surrogate characterization of CRM in $E$ from Theorems~\ref{thm:gradientdescent} and \ref{thm:quadratic} illuminates the geometry of CRM well beyond the limited analysis provided in \cite{DHL2019}, and it provides an explicit bridge to state-of-the-art theoretical results for nonconvex DR. Future works may now use Lyapunov functions to study not only the broadly usable method $L_T$, but the feasibility method CRM as well.

Convergence guarantees for a class of problems as broad as those considered here do not exist for any algorithm. Even for the special case of nonconvex feasibility problems and the Douglas--Rachford operator, such results are few, require significant structure on the objective \cite{AB,Benoist,BLSSS,DT}, and are quite complicated to prove. For the more general nonconvex optimization problem, the broadest is the work of Li and Pong \cite{LP}. Even the ``specific'' $\mathbf{MSS}(V)$ candidate we demonstrated with, $L_T$, is defined for a general operator $T$, uses multiple steps of $T$ in its construction, and may take bolder steps than $T$ does. For all of these reasons, convergence guarantees that involve $L_T$---and other members of $\mathbf{MSS}(V)$---will almost certainly have to consider the specific structure not only of the objective, but of the chosen operator $T$ as well. Two perfect examples of such specific $T$ present themselves, and we recommend these as yet two more specific steps of investigation.

Firstly, we already have the finite convergence of $L_T$ when $T$ is the Douglas--Rachford operator for the feasibility problem of two lines in $\R^n$. This, together with the outstanding performance of $L_T$ for ADMM with the basis pursuit problem, suggests that this analysis should be extended for DR for more general affine and convex feasibility problems. 

Secondly, \cite[Appendix A]{boyd2011distributed} provides a proof of convergence for ADMM by means of a dilated quadratic Lyapunov function. An important future work is to study whether $L_T$ minimizes quadratic surrogates related to this \textit{specific} Lyapunov function, or whether a new, analogous algorithm based upon this Lyapunov function may be designed.

\subsubsection*{Acknowledgements}

The author was supported by Hong Kong Research Grants Council PolyU153085/16p and by the Alf van der Poorten Traveling Fellowship (Australian Mathematical Society). The author especially thanks Brailey Sims for his careful reading and many helpful comments.

\bibliographystyle{plain}
\bibliography{bibliography}

\end{document}